\documentclass[reqno]{amsart}

\usepackage{amsmath, amsfonts, amssymb, hyperref,tikz}
\usepackage{graphicx}

\usetikzlibrary{matrix,arrows,shapes}
\usetikzlibrary{backgrounds}

{\theoremstyle{plain}
\newtheorem{theorem}{Theorem}[section]
\newtheorem{conjecture}[theorem]{Conjecture}
\newtheorem{proposition}[theorem]{Proposition}
\newtheorem{lemma}[theorem]{Lemma}
\newtheorem{corollary}[theorem]{Corollary}
\newtheorem{question}[theorem]{Question}
}
{\theoremstyle{definition}
\newtheorem{definition}[theorem]{Definition}
\newtheorem*{definition*}{Definition}
\newtheorem{remark}[theorem]{Remark}
\newtheorem{remarks}[theorem]{Remarks}
\newtheorem{example}[theorem]{Example}
}

\numberwithin{equation}{section}
\numberwithin{figure}{section}

\newcommand{\gen}[1]{K_{(#1)}}
\newcommand{\genx}[1]{K_{#1}}
\newcommand{\om}{\omega} 
\newcommand{\kpo}{\gen{P,\om}}
\newcommand{\kqt}{\gen{Q,\tau}}
\newcommand{\pop}{f}
\newcommand{\pcovered}{\preccurlyeq_P}
\newcommand{\qsym}{g}
\newcommand{\des}{\mathrm{Des}}
\newcommand{\rotate}[1]{(#1)^*}
\newcommand{\switch}[1]{\overline{(#1)}}
\newcommand{\switchx}[1]{\overline{#1}}
\newcommand{\co}{\mathrm{co}}
\newcommand{\jump}{\mathrm{jump}}
\newcommand{\anti}{\mathrm{anti}}
\newcommand{\nenwarrows}{\nearrow\!\!\!\!\!\!\nwarrow}
\newcommand{\NeNwarrows}{\Nearrow\!\!\!\!\!\!\Nwarrow}
\newcommand{\neNwarrows}{\nearrow\!\!\!\!\!\!\Nwarrow}
\newcommand{\du}[4]{(#1,#2)+(#3,#4)}
\newcommand{\up}[4]{(#1,#2)\uparrow(#3,#4)}
\newcommand{\Up}[4]{(#1,#2)\Uparrow(#3,#4)}
\newcommand{\neast}[4]{(#1,#2)\nearrow(#3,#4)}
\newcommand{\Ne}[4]{(#1,#2)\Nearrow(#3,#4)}
\newcommand{\nenw}[4]{(#1,#2)\nenwarrows(#3,#4)}
\newcommand{\NeNw}[4]{(#1,#2)\NeNwarrows(#3,#4)}
\newcommand{\neNw}[4]{(#1,#2)\neNwarrows(#3,#4)}
\newcommand{\switchnenw}[4]{\switch{#1,#2}\nenwarrows\switch{#3,#4}}
\newcommand{\stwoone}{$s_{21}$}
\newcommand{\stwooneone}{$s_{211}$}
\newcommand{\single}{\circ}
\newcommand{\sd}[1]{P_{#1}}
\newcommand{\twoplusone}{\mathbf{2}+\mathbf{1}}
\newcommand{\oneplusone}{\mathbf{1}+\mathbf{1}}

\DeclareSymbolFont{symbolsC}{U}{txsyc}{m}{n}
\SetSymbolFont{symbolsC}{bold}{U}{txsyc}{bx}{n}
\DeclareFontSubstitution{U}{txsyc}{m}{n}
\DeclareMathSymbol{\Nearrow}{\mathrel}{symbolsC}{116}
\DeclareMathSymbol{\Searrow}{\mathrel}{symbolsC}{117}
\DeclareMathSymbol{\Nwarrow}{\mathrel}{symbolsC}{118}
\DeclareMathSymbol{\Swarrow}{\mathrel}{symbolsC}{119}

\begin{document}
\title{Equality of $P$-partition generating functions}

\author{Peter R.\,W. McNamara}

\address{Department of Mathematics, Bucknell University, Lewisburg, PA 17837, USA}
\email{\href{mailto:peter.mcnamara@bucknell.edu}{peter.mcnamara@bucknell.edu}}

\author{Ryan E. Ward}
\address{Department of Mathematics, The Pennsylvania State University, University Park, PA 16802, USA}
\email{\href{mailto:ward@math.psu.edu}{ward@math.psu.edu}}

\begin{abstract}
To every labeled poset $(P,\om)$, one can associate a quasisymmetric generating function for its $(P,\om)$-partitions.  We ask: when do two labeled posets have the same generating function?  Since the special case corresponding to skew Schur function equality is still open, a complete classification of equality among $(P,\om)$ generating functions is likely too much to expect.   Instead, we determine necessary conditions and separate sufficient conditions for two labeled posets to have equal generating functions.  We conclude with a classification of all equalities for labeled posets with small numbers of linear extensions.  
\end{abstract}

\maketitle


\section{Introduction}\label{sec:intro} 
 
Because of their well-documented and well-known connections with other areas of mathematics, Schur functions and Littlewood--Richardson coefficients are central objects of study in algebraic combinatorics.  A natural generalization of the class of Schur functions is the class of skew Schur functions, and these are also intimately connected to Littlewood--Richardson coefficients: 
specifically, Littlewood-Richardson coefficients can be defined as the coefficients that result from expanding a skew Schur function in the basis of Schur functions.
Since skew Schur functions are too abundant to comprise a basis for the algebra of symmetric functions, it is natural to study relationships among them.  Skew Schur functions are indexed by skew diagrams, and so an obvious first step is to ask when two skew diagrams correspond to equal skew Schur functions.  This question has received considerable attention  in recent years \cite{BTvW06,Gut09,McvW09,RSvW07}, but remains wide open.  In fact, building on \cite{BTvW06,RSvW07}, a conjecture is given in \cite{McvW09} for necessary and sufficient conditions for two skew diagrams to give the same skew Schur function, yet both directions of the conjecture are open.  

Sometimes a way to gain insight into a problem is to consider a more general one, which was one of the motivating ideas behind the present work.  A broad but natural generalization of a skew diagram $\lambda/\mu$ is a labeled poset $(P,\om)$.  Under this extension, semistandard Young tableaux of shape $\lambda/\mu$ are generalized to $(P,\om)$-partitions, as first defined in \cite{StaThesis71,StaThesis}.  Moreover, the skew Schur function $s_{\lambda/\mu}$ is generalized to the generating function $\kpo$ for $(P,\om)$-partitions.  As we will see, $\kpo$ is a quasisymmetric function;  see \cite{Ges84} and \cite[\S 7.19]{ec2} for further information about $\kpo$.  
(We note here that $\kpo$ is different from another generating function for $(P,\om)$-partitions, namely $F_{P,\om}$ of \cite[\S 3.15]{ec1e2} or its special case $F_P$ of \cite[\S 4.5]{ec1}.  In particular, the poset $P$ can be recovered from $F_P$ or $F_{P,\om}$, which is certainly not the case with $\kpo$ as we shall see in detail.)  

Our goal, therefore, is to determine conditions on labeled posets $(P,\om)$ and $(Q,\tau)$ for $\kpo = \kqt$.  If the latter equality is true, we will write $(P,\om) \sim (Q,\tau)$. A nontrivial example of such labeled posets is given in Figure~\ref{fig:nontrivial}.

\begin{figure}[htbp]
\begin{center}
\begin{tikzpicture}[scale=1.0]
\begin{scope}
\tikzstyle{every node}=[draw,shape=circle, inner sep=2pt]; 
\draw (0,0) node (a1) {1}; 
\draw (1,0) node (a2) {3}; 
\draw (0,1) node (a3) {4};
\draw (1,1) node (a4) {2};
\draw (1,2) node (a5) {5};
\draw (a1) -- (a3)
(a1) -- (a4)
(a2) -- (a3)
(a2) -- (a4)
(a4) -- (a5); 
\draw (3.5,0) node (b1) {2}; 
\draw (3,1) node (b2) {1}; 
\draw (4,1) node (b3) {3};
\draw (3,2) node (b4) {4};
\draw (4,2) node (b5) {5};
\draw (b1) -- (b2)
(b1) -- (b3)
(b2) -- (b4)
(b3) -- (b5);
\end{scope}
\draw (2,1) node {$\sim$};
\end{tikzpicture}
\caption{Two labeled posets that have the same $(P,\om)$-partition generating function $\kpo$.}
\label{fig:nontrivial}
\end{center}
\end{figure}
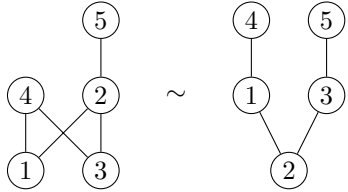

Not surprisingly, since the special case involving skew Schur functions remains open, we are unable to determine conditions on labeled posets that are both necessary and sufficient for $\kpo$-equality.  In fact, there are many interesting examples of $\kpo$-equality that have no analogue in the skew Schur function case, with Figure~\ref{fig:nontrivial} being one such example.  Instead, we address $\kpo$-equality by giving some conditions that are necessary and some that are sufficient.  Because this is the first paper on $\kpo$-equality, we benefit from the opportunity to record some results that can be considered as ``low-hanging fruit.''  We also present some more substantial results, but have not aimed to be exhaustive, and we hope that others will be attracted to the area by the wealth of unanswered questions.

The rest of this paper is organized as follows.  In Section~\ref{sec:prelim}, we give the necessary background, including the definition of $\kpo$.  In Section \ref{sec:initial}, we make some initial observations pertaining to the equality question, some of which follow directly from the definition of $\kpo$ and others of which come from ideas already appearing in the literature.  In Section \ref{sec:morenec}, we give necessary conditions for $\kpo$-equality based around the idea of the \emph{jump sequence} of a labeled poset.  Section \ref{sec:suff} gives sufficient conditions for equality in the form of several operations that can be performed on labeled posets that preserve $\kpo$-equality.  In Section \ref{sec:fewlinext}, we give conditions that are both necessary and sufficient for equality in the case of posets with at most three linear extensions; we show that all such equalities arise from simple equalities among skew Schur functions.  We conclude in Section~\ref{sec:open} with open problems.  

\subsection*{Acknowledgements} We thank Christophe Reutenauer for posing the question of $\kpo$-equality to the first author, and the anonymous referees for their careful reading of the manuscript.  Portions of this paper were written while the first author was on sabbatical at Trinity College Dublin; he thanks the School of Mathematics for its hospitality.  Computations were performed using Sage \cite{sage}.


\section{Preliminaries}\label{sec:prelim}

In this section, we give the necessary background information about labeled posets, $(P,\om)$-partitions, $\kpo$, and quasisymmetric functions; more details can be found in \cite{Ges84} and \cite[\S 7.19]{ec2}.

\subsection{Labeled posets and $(P,\om)$-partitions}
All our posets $P$ will be finite with cardinality denoted $|P|$.  Let $[n]$ denote the set $\{1, \ldots, n\}$.  The order relation on $P$ will be denoted $\leq_P$, while $\leq$ will denote the usual order on the positive integers.  A \emph{labeling} of a poset is a bijection $\om : P \to [|P|]$.  A \emph{labeled poset} $(P,\om)$ is then a poset $P$ with an associated labeling $\om$.  

\begin{definition}\label{def:popartition}
For a labeled poset $(P,\om)$, a $(P,\om)$-partition is a map $\pop$ from $P$ to the positive integers satisfying the following two conditions:
\begin{itemize}
\item if $a \leq_P b$, then $\pop(a) \leq \pop(b)$, i.e., $\pop$ is order-preserving;
\item if $a \leq_P b$ and $\om(a) > \om(b)$, then $\pop(a) < \pop(b)$.
\end{itemize}
\end{definition}
In other words, a $(P,\om)$-partition is an order-preserving map from $P$ to the positive integers with certain strictness conditions determined by $\om$.  
Examples of $(P,\om)$-partitions are given in Figure~\ref{fig:popartitions}, where the images under $f$ are written in bold next to the nodes.  The meaning of the double edges follows from the following observation about Definition~\ref{def:popartition}.  For $a, b \in P$, we say that $a$ is \emph{covered} by $b$ in $P$, denoted $a \pcovered b$, if $a \leq_P b$ and there does not exist $c$ in $P$ such that $a <_P c <_P b$.  Note that a definition equivalent to Definition~\ref{def:popartition} is obtained by replacing both appearances of the relation $a \leq_P b$ with the relation $a \pcovered b$.  In other words, we require that $\pop$ be order-preserving along the edges of the Hasse diagram of $P$, with $\pop(a) < \pop(b)$ when $a \pcovered b$ with $\om(a) > \om(b)$.  With this in mind, we will consider those edges $a \pcovered b$ with $\om(a) > \om(b)$ as \emph{strict edges} and we will represent them in Hasse diagrams by double lines.  Similarly, edges $a \pcovered b$ with $\om(a) < \om(b)$ will be called \emph{weak edges} and will be represented by single lines. 
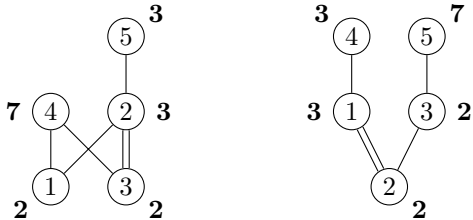
\begin{figure}[htbp]
\begin{center}
\begin{tikzpicture}[scale=1.0]
\begin{scope}
\tikzstyle{every node}=[shape=circle, inner sep=2pt]; 
\draw (0,0) node[draw] (a1) {1} +(-0.4,-0.3) node {\textbf{2}}; 
\draw (1,0) node[draw] (a2) {3} +(0.4,-0.3) node {\textbf{2}};
\draw (0,1) node[draw] (a3) {4} +(-0.5,0) node {\textbf{7}};
\draw (1,1) node[draw] (a4) {2} +(0.5,0) node {\textbf{3}};
\draw (1,2) node[draw] (a5) {5} +(0.4,0.3) node {\textbf{3}};
\draw[double distance=2pt] (a2) -- (a4);
\draw (a1) -- (a3)
(a1) -- (a4)
(a2) -- (a3)
(a4) -- (a5); 
\begin{scope}[xshift=1cm]
\draw (3.5,0) node[draw] (b1) {2} +(0.4,-0.3) node {\textbf{2}}; 
\draw (3,1) node[draw] (b2) {1} +(-0.5,0) node {\textbf{3}}; 
\draw (4,1) node[draw] (b3) {3} +(0.5,0) node {\textbf{2}};
\draw (3,2) node[draw] (b4) {4} +(-0.4,0.3) node {\textbf{3}};
\draw (4,2) node[draw] (b5) {5} +(0.4,0.3) node {\textbf{7}};
\draw[double distance=2pt] (b1) -- (b2);
\draw
(b1) -- (b3)
(b2) -- (b4)
(b3) -- (b5);
\end{scope}
\end{scope}
\end{tikzpicture}
\caption{Two examples of $(P,\om)$-partitions.}
\label{fig:popartitions}
\end{center}
\end{figure}

From the point-of-view of $(P,\om)$-partitions, the labeling $\om$ only determines which edges are strict and which are weak.  Therefore, many of our figures from this point on will not show the labeling $\om$, but instead show some collection of strict and weak edges determined by an underlying $\om$.  Furthermore, it will make many of our explanations simpler if we think of $\om$ as an assignment of strict and weak edges, rather than as a labeling of the elements of $P$, especially in the later sections. For example, we say that labeled posets $(P,\om)$ and $(Q,\tau)$ are \emph{isomorphic}, written $(P,\om) \cong (Q,\tau)$ if there exists a poset isomorphism from $P$ to $Q$ that sends strict (respectively weak) edges to strict (resp.\ weak) edges.  We will be careful to refer to the underlying labels when necessary.  

If all the edges are weak, then $P$ is said to be \emph{naturally labeled}.  In this case, a $(P,\om)$-partition is traditionally called a $P$-partition (although sometimes the term ``$P$-partition'' is used informally as an abbreviation for ``$(P,\om)$-partition,'' as in the title of this paper).  Note that if $P$ is a naturally-labeled chain, then a $P$-partition gives a partition of an integer, and generalizing the theory of partitions was a motivation for Stanley's definition of $(P,\om)$-partitions \cite{StaThesis}.

\subsection{The generating function $\kpo$}

Using $x$ to denote the sequence of variables $x_1, x_2, \ldots$, we can now define our main object of study $\kpo$.  

\begin{definition}\label{def:kpo}
For a labeled poset $(P,\om)$, we define the $(P,\om)$-partition generating function $\kpo = \kpo(x)$ by
\[
\kpo(x) = \sum_{(P,\om)\textnormal{-partition }f} x_1^{|\pop^{-1}(1)|} x_2^{|\pop^{-1}(2)|} \cdots,
\]
where the sum is over all $(P,\om)$-partitions $\pop$.
\end{definition}

For example, the $(P,\om)$-partitions of Figure~\ref{fig:popartitions} would each contribute the monomial $x_2^2 x_3^2 x_7$ to their respective $\kpo$.  As another simple example, if $(P,\om)$ is a naturally-labeled chain with three elements, then $\kpo = \sum_{i \leq j \leq k} x_i x_j x_k$.

\begin{example}\label{exa:skew}
Given a skew diagram $\lambda/\mu$ in French notation with $n$ cells, label the cells with the numbers $[n]$ in any way that makes the labels increase down columns and from left to right along rows, as in Figure~\ref{fig:skewschur}(a).   Rotating the result 45$^\circ$ in a counter-clockwise direction and replacing the cells by nodes as in Figure~\ref{fig:skewschur}(b), we get a corresponding labeled poset which we denote by $(\sd{\lambda/\mu}, \om)$ and call a \emph{skew-diagram labeled poset}.  Under this construction, we see that a $(\sd{\lambda/\mu}, \om)$-partition corresponds exactly to a semistandard Young tableau of shape $\lambda/\mu$. Therefore $\gen{\sd{\lambda/\mu}, \om}$ is exactly the skew Schur function $s_{\lambda/\mu}$,
justifying the claim in the introduction that the study of $\kpo$-equality is a generalization of the study of skew Schur equality.  
\begin{figure}[htbp]
\begin{center}
\begin{tikzpicture}[scale=0.7]
\begin{scope}[yshift=-0.4cm]
\draw[thick] (0,3) -- (3,3) -- (3,2) -- (4,2) -- (4,0) -- (2,0) -- (2,1) -- (1,1) -- (1,2) -- (0,2) -- cycle;
\draw (1,2) -- (3,2) 
(2,1) -- (4,1) 
(1,2) -- (1,3) 
(2,1) -- (2,3) 
(3,0) -- (3,2);
\draw[dashed] (2,0) -- (0,0) -- (0,2)
(1,0) -- (1,1) -- (0,1);
\begin{scope}[font = \Large]
\draw (0.5,2.5) node {1};
\draw (1.5,2.5) node {2};
\draw (2.5,2.5) node {3};
\draw (1.5,1.5) node {4};
\draw (2.5,1.5) node {5};
\draw (2.5,0.5) node {6};
\draw (3.5,1.5) node {7};
\draw (3.5,0.5) node {8};
\end{scope}
\end{scope}
\draw (2,-1) node {(a)};
\begin{scope}[xshift=6cm]
\begin{scope} 
\tikzstyle{every node}=[draw, shape=circle, inner sep=2pt]; 
\draw (0,0) node (a1) {1};
\draw (1,1) node (a2) {2};
\draw (2,2) node (a3) {3};
\draw (2,0) node (a4) {4};
\draw (3,1) node (a5) {5};
\draw (4,0) node (a6) {6};
\draw (4,2) node (a7) {7};
\draw (5,1) node (a8) {8};
\draw[double distance=2pt] (a4) -- (a2)
(a6) -- (a5) -- (a3)
(a8) -- (a7);
\draw (a1) -- (a2) --(a3)
(a4) -- (a5) -- (a7)
(a6) -- (a8); 
\end{scope}
\draw (2.5,-1) node {(b)};
\end{scope}
\end{tikzpicture}
\caption{The skew diagram $443/21$ and a corresponding labeled poset.}
\label{fig:skewschur}
\end{center}
\end{figure}
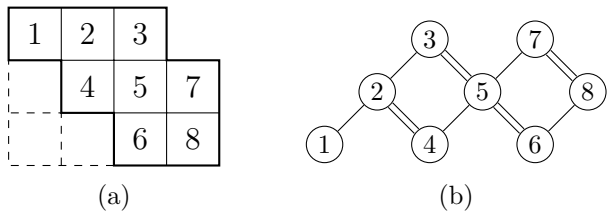
\end{example}

\subsection{Quasisymmetric functions}

It follows directly from the definition of quasisymmetric functions below that $\kpo$ is quasisymmetric.  In fact, $\kpo$ served as motivation for Gessel's original definition \cite{Ges84} of quasisymmetric functions.  For a formal power series $\qsym$ in the variables $x_1, x_2, \ldots$, let $[x_{i_1}^{a_1} x_{i_2}^{a_2} \cdots x_{i_k}^{a_k}]\, \qsym$ denote the coefficient of $x_{i_1}^{a_1} x_{i_2}^{a_2} \cdots x_{i_k}^{a_k}$ in the expansion $\qsym$ into monomials.

\begin{definition}
A quasisymmetric function in the variables $x_1, x_2,\ldots$, say with rational coefficients, is a formal power series $\qsym = \qsym(x) \in \mathbb{Q}[[x_1, x_2, \ldots]]$ of bounded degree such that for every sequence $a_1, a_2, \ldots a_k$ of positive integer exponents, we have
\[
[x_{i_1}^{a_1} x_{i_2}^{a_2} \cdots x_{i_k}^{a_k}]\, \qsym
= [x_{j_1}^{a_1} x_{j_2}^{a_2} \cdots x_{j_k}^{a_k}]\, \qsym
\]
whenever $i_1 < i_2 < \cdots < i_k$ and $j_1 < j_2 < \cdots < j_k$.
\end{definition}
Notice that we get the definition of a symmetric function if we replace the condition that the sequences $i_1, i_2, \ldots i_k$ and $j_1, j_2, \ldots, j_k$ be strictly increasing with the weaker condition that each sequence consists of distinct elements.  As an example, the formal power series
\[
\sum_{1 \leq i < j} x_i^2 x_j
\]
is quasisymmetric but not symmetric.  

In our study of $\kpo$-equality, we will make use of both of the classical bases for the vector space of quasisymmetric functions.  If $\alpha = (\alpha_1, \alpha_2, \ldots, \alpha_k)$ is a composition of $n$, then we define the \emph{monomial quasisymmetric function} $M_\alpha$ by
\[
M_\alpha = \sum_{i_1 < i_2 < \ldots < i_k} x_{i_1}^{\alpha_1} x_{i_2}^{\alpha_2} \cdots x_{i_k}^{\alpha_k}.
\]
It is clear that the set $\{M_\alpha\}$, where $\alpha$ ranges over all compositions of $n$, forms a basis for the vector space of quasisymmetric functions of degree $n$.  As we know, compositions of $n$ are in bijection with subsets of $[n-1]$, and let $S(\alpha)$ denote the set
$\{\alpha_1, \alpha_1+\alpha_2, \ldots, \alpha_1+\alpha_2 + \cdots \alpha_{k-1}\}$.  Furthermore, for a subset $S$ of $[n-1]$, we write $\co(S)$ for the composition $\alpha$ of $n$ satisfying $S(\alpha)=S$.  It will be helpful to sometimes denote $M_\alpha$ by $M_{S(\alpha), n}$.  Notice that these two notations are distinguished by the latter one including the subscript $n$, which is helpful due to  $S(\alpha)$ not uniquely determining $n$.  

The monomial quasisymmetric basis is natural enough that we can see directly from Definition~\ref{def:kpo} how to expand $\kpo$ in this basis.  Indeed, for the composition $\alpha = (\alpha_1, \ldots, \alpha_k)$, the coefficient of $M_\alpha$ in $\kpo$ will be the number of $(P,\om)$-partitions $\pop$ such that $|\pop^{-1}(1)| = \alpha_1$, \ldots, $|\pop^{-1}(k)| = \alpha_k$.  

Again for $\alpha$ a composition of $n$, the second classical basis is composed of the \emph{fundamental quasisymmetric functions} $F_\alpha$, defined by 
\[
F_\alpha = F_{S(\alpha), n} = \sum_{S(\alpha) \subseteq T  \subseteq [n-1]} M_{T,n}.
\]
The relevance of this latter basis to $\kpo$ is due to Theorem~\ref{thm:kexpansion} below, which first appeared in \cite{StaThesis71,StaThesis}, although the first appearance in the language of quasisymmetric functions is in \cite{Ges84}.  Every permutation $\pi \in S_n$ has a descent set $\des(\pi)$ given by $\{i \in [n-1] : \pi(i) > \pi(i+1)\}$, and we will call $\co(\des(\pi))$ the \emph{descent composition} of $\pi$.  Let $\mathcal{L}(P,\om)$ denote the set of all linear extensions of $P$, regarded as permutations of the $\om$-labels of $P$.  For example, for the labeled poset of Figure~\ref{fig:stathm}(b), $\mathcal{L}(P,\om) = \{132, 312\}$.  

\begin{theorem}[\cite{Ges84,StaThesis71,StaThesis}]\label{thm:kexpansion}
Let $(P,\om)$ be a labeled poset with $|P|=n$.  Then
\[
\kpo = \sum_{\pi \in \mathcal{L}(P,\om)} F_{\des(\pi), n}.
\]
\end{theorem}

\begin{example}
The labeled poset $(P_1,\om_1)$ of Figure~\ref{fig:stathm}(a) has $\mathcal{L}(P,\om) = \{213, 231\}$ and hence
\begin{align*}
\gen{P_1, \om_1} &= F_{\{1\},3} + F_{\{2\},3} \\
&=  F_{12} + F_{21}\\ 
&=  M_{\{1\},3} + M_{\{2\},3} + 2M_{\{1,2\},3}\\
&=  M_{12} + M_{21} + 2M_{111}.
\end{align*}

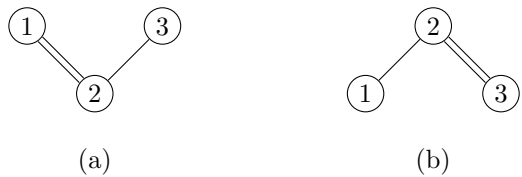
\begin{figure}[htbp]
\begin{center}
\begin{tikzpicture}[scale=0.9]
\begin{scope}
\tikzstyle{every node}=[draw,shape=circle, inner sep=2pt]; 
\draw (1,0) node (a1) {2}; 
\draw (0,1) node (a2) {1}; 
\draw (2,1) node (a3) {3};
\draw[double distance=2pt](a1) -- (a2);
\draw (a1) -- (a3);
\end{scope}
\draw (1,-1) node {(a)};
\begin{scope}[xshift=5cm]
\begin{scope}
\tikzstyle{every node}=[draw,shape=circle, inner sep=2pt]; 
\draw (0,0) node (a1) {1}; 
\draw (2,0) node (a2) {3}; 
\draw (1,1) node (a3) {2};
\draw (a1) -- (a3);
\draw[double distance=2pt](a2) -- (a3);
\end{scope}
\draw (1,-1) node {(b)};
\end{scope}
\end{tikzpicture}
\caption{The posets of least cardinality with nontrivially equal $\kpo$.}
\label{fig:stathm}
\end{center}
\end{figure}

We conclude this section with the nice observation that the labeled poset $(P_2,\om_2)$ of Figure~\ref{fig:stathm}(b) satisfies 
\begin{equation}\label{equ:skewequiv}
(P_1,\om_1) \sim (P_2,\om_2),
\end{equation}
and this is the smallest nontrivial example of $\kpo$-equality.  Moreover, it will serve as a building block for larger $\kpo$-equalities, particularly in Section~\ref{sec:fewlinext}.  That $(P_1,\om_1) \sim (P_2, \om_2)$ is no surprise once we observe that $(P_1,\om_1)$ (resp.\ $(P_2,\om_2)$) comes from the skew diagram 21 (resp.\ 22/1) under the construction of Example~\ref{exa:skew}; it is well-known (e.g.\ \cite[Exer.\ 7.56(a)]{ec2}) that skew Schur functions are invariant under a 180$^\circ$ rotation of their skew diagrams.  
\end{example}


\section{The equality question: initial observations}\label{sec:initial}

In this section we collect some basic results about conditions on labeled posets $(P,\om)$ and $(Q,\tau)$ for $(P,\om) \sim (Q,\tau)$.  

The first observation follows directly from the definition of $\kpo$, Definition~\ref{def:kpo}.
\begin{proposition}\label{pro:cardinality}
If labeled posets $(P,\om)$ and $(Q,\tau)$ satisfy $(P,\om) \sim (Q,\tau)$, then $|P|=|Q|$.  
\end{proposition}

Theorem~\ref{thm:kexpansion}, along with the fact that the fundamental quasisymmetric functions form a basis for the quasisymmetric functions, gives another necessary condition.
\begin{proposition}\label{pro:linexts}
If labeled posets $(P,\om)$ and $(Q,\tau)$ satisfy $(P,\om) \sim (Q,\tau)$, then $|\mathcal{L}(P,\om)| = |\mathcal{L}(Q,\tau)|$.
\end{proposition}

The next result answers the question of whether a naturally labeled poset can have the same generating function as a non-naturally labeled poset.

\begin{proposition}
Suppose $(P,\om)$ is a naturally labeled poset and $(P,\om) \sim (Q,\tau)$ for some labeled poset $(Q,\tau)$.  Then $(Q,\tau)$ is also naturally labeled.
\end{proposition}

\begin{proof}
Observe that $(P,\om)$ is naturally labeled if and only if the identity permutation is an element of $\mathcal{L}(P,\om)$.  By Theorem~\ref{thm:kexpansion}, the latter condition is equivalent to $F_{\emptyset, |P|}$ appearing in the $F$-expansion of $\kpo$, and the result follows.  
\end{proof}

One can check that the vector space of quasisymmetric functions is an algebra \cite[Exer.\ 7.93]{ec2}, and one might wonder about the meaning of the product of two $\kpo$ generating functions.  
The answer is as nice as one might hope.  Given two labeled posets $(P_1,\om_1)$ and $(P_2,\om_2)$, we define the disjoint union $(P,\om) = \du{P_1}{\om_1}{P_2}{\om_2}$ as follows.  As usual, $P$ is just the disjoint union of $P_1$ and $P_2$.  For $a \in P$, the labeling $\om$ is defined by $\om|_{P_1}(a) = \om_1(a)$ and $\om|_{P_2}(a) = \om_2(a) + |P_1|$.  In other words, the labels of $P_1$ are preserved, and the labels of $P_2$ are all increased by $|P_1|$ (cf.\ \cite[I.12]{StaThesis}), as in Figure~\ref{fig:product}. 

\begin{proposition}\label{pro:product}
For labeled posets $(P_1,\om_1)$ and $(P_2,\om_2)$, we have 
\[
\genx{\du{P_1}{\om_1}{P_2}{\om_2}} = \gen{P_1,\om_1} \gen{P_2,\om_2}\,.
\]
\end{proposition}

\begin{proof}
Since the strict and weak edges of $\du{P_1}{\om_1}{P_2}{\om_2}$ remain as they were in $(P_1,\om_1)$ and $(P_2,\om_2)$, there is an obvious correspondence 
between $(\du{P_1}{\om_1}{P_2}{\om_2})$-partitions and pairs consisting of a $(P_1,\om_1)$-partition and a $(P_2,\om_2)$-parti\-tion, as in the example of Figure~\ref{fig:product}.  The result now follows by considering the expansion of the generating functions in terms of monomials, as given in Definition~\ref{def:kpo}.
\end{proof}

\begin{figure}[htbp]
\begin{center}
\begin{tikzpicture}[scale=0.9]
\begin{scope}
\tikzstyle{every node}=[shape=circle, inner sep=2pt]; 
\draw (1,0) node[draw] (a1) {2} + (0.4,-0.3) node {\textbf{7}}; 
\draw (0,1) node[draw] (a2) {1} + (-0.5,0) node {\textbf{9}}; 
\draw (2,1) node[draw] (a3) {3} + (-0.5,0) node {\textbf{7}};
\draw[double distance=2pt](a1) -- (a2);
\draw (a1) -- (a3);
\begin{scope}[xshift=3cm]
\tikzstyle{every node}=[shape=circle, inner sep=2pt]; 
\draw (0,0) node[draw] (b1) {5} + (0.5,0) node {\textbf{4}}; 
\draw (0,1) node[draw] (b2) {4} + (0.5,0) node {\textbf{6}}; 
\draw[double distance=2pt](b1) -- (b2);
\end{scope}
\end{scope}
\begin{scope}[xshift=5cm,yshift=5mm]
\draw (0,0) node {$\longleftrightarrow$};
\end{scope}
\begin{scope}[xshift=7cm]
\tikzstyle{every node}=[shape=circle, inner sep=2pt]; 
\draw (1,0) node[draw] (a1) {2} + (0.4,-0.3) node {\textbf{7}}; 
\draw (0,1) node[draw] (a2) {1} + (-0.5,0) node {\textbf{9}}; 
\draw (2,1) node[draw] (a3) {3} + (-0.5,0) node {\textbf{7}};
\draw[double distance=2pt](a1) -- (a2);
\draw (a1) -- (a3);
\end{scope}
\begin{scope}[xshift=11cm]
\begin{scope}
\tikzstyle{every node}=[shape=circle, inner sep=2pt]; 
\draw (0,0) node[draw] (b1) {2} + (0.5,0) node {\textbf{4}}; 
\draw (0,1) node[draw] (b2) {1} + (0.5,0) node {\textbf{6}}; 
\draw[double distance=2pt](b1) -- (b2);
\end{scope}
\end{scope}
\end{tikzpicture}
\caption{The $(\du{P_1}{\om_1}{P_2}{\om_2})$-partition on the left is comprised of the $(P_1,\om_1)$-partition and $(P_2,\om_2)$-partition on the right.}
\label{fig:product}
\end{center}
\end{figure}
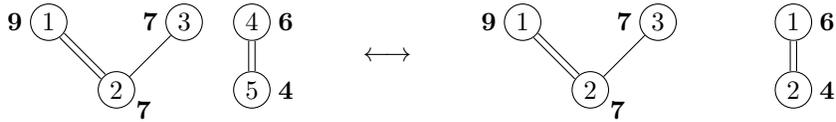

For the rest of this section, we consider some natural involutions we can perform on a labeled poset $(P,\om)$. These involutions appear previously in, for example, \cite{Ehr96,Ges90,LaPy08,MalThesis, MaRe95, MaRe98} and \cite[Exer.~7.94(a)]{ec2}.  First, we can switch strict and weak edges, denoting the result $\switch{P,\om}$. Secondly, we can rotate the labeled poset 180$^\circ$, preserving strictness and weakness of edges; we denote the resulting labeled poset $\rotate{P,\om}$.  Observe that these bar and star involutions commute; an example is given in Figure~\ref{fig:involutions}.    

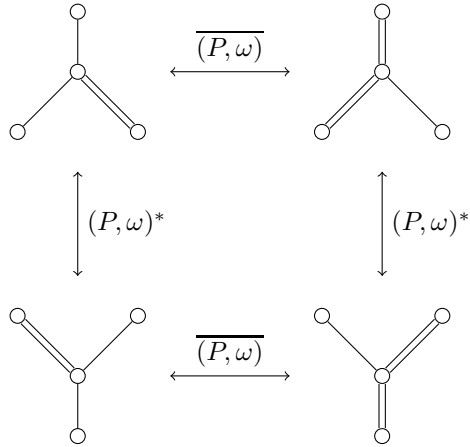
\begin{figure}[htbp]
\begin{center}
\begin{tikzpicture}[scale=1.0]
\matrix [row sep=20mm, column sep=20mm]  {
\node(NW) {
\begin{tikzpicture}[scale=0.8]
\tikzstyle{every node}=[draw, shape=circle, inner sep=2pt]; 
\draw (0,0) node (a1) {};
\draw (2,0) node (a2) {};
\draw (1,1) node (a3) {};
\draw (1,2) node (a4) {};
\draw (a1) -- (a3)
(a3) -- (a4);
\draw[double distance=2pt] (a2) -- (a3);
\end{tikzpicture}
};
&
\node(NE) {
\begin{tikzpicture}[scale=0.8]
\tikzstyle{every node}=[draw, shape=circle, inner sep=2pt]; 
\draw (0,0) node (b1) {};
\draw (2,0) node (b2) {};
\draw (1,1) node (b3) {};
\draw (1,2) node (b4) {};
\draw[double distance=2pt] (b1) -- (b3)
(b3) -- (b4);
\draw (b2) -- (b3); 
\end{tikzpicture}
};
\\
\node(SW) {
\begin{tikzpicture}[scale=0.8]
\tikzstyle{every node}=[draw, shape=circle, inner sep=2pt]; 
\draw (1,0) node (c1) {};
\draw (1,1) node (c2) {};
\draw (0,2) node (c3) {};
\draw (2,2) node (c4) {};
\draw (c1) -- (c2)
(c2) -- (c4);
\draw[double distance=2pt](c2) -- (c3);
\end{tikzpicture}
};
&
\node(SE) {
\begin{tikzpicture}[scale=0.8]
\tikzstyle{every node}=[draw, shape=circle, inner sep=2pt]; 
\draw (1,0) node (d1) {};
\draw (1,1) node (d2) {};
\draw (0,2) node (d3) {};
\draw (2,2) node (d4) {};
\draw (d2) -- (d3);
\draw[double distance=2pt](d1) -- (d2)
(d2) -- (d4);
\end{tikzpicture}
};
\\ }; 
\draw[<->,shorten >=2mm, shorten <= 2mm] (NW) -- (NE) node [midway,auto] {$\switch{P,\om}$};
\draw[<->,shorten >=3mm, shorten <=3mm] (NW) -- (SW) node [midway,auto] {$\rotate{P,\om}$};
\draw[<->,shorten >=3mm, shorten <=3mm] (NE) -- (SE) node [midway,auto] {$\rotate{P,\om}$};
\draw[<->,shorten >=2mm, shorten <= 2mm] (SW) -- (SE) node [midway,auto] {$\switch{P,\om}$};
\end{tikzpicture}
\caption{The bar and star involutions.}
\label{fig:involutions}
\end{center}
\end{figure}

Although in subsequent sections we will prefer to view these bar and star involutions as described above, it will be helpful for Lemma~\ref{lem:involutions} below to formulate them in terms of their effect on the $\om$-labels.  
We see that 
\begin{equation}\label{equ:barisom}
\switch{P,\om} \cong (P,\om'),
\end{equation}
where $\om'$ is defined by 
\begin{equation}\label{equ:bar}
\om'(a) = |P|+1-\om(a)
\end{equation}
for all $a \in P$.  We also have 
\begin{equation}\label{equ:starisom}
\rotate{P,\om} \cong (P^*,\om'),
\end{equation}
where $P^*$ is the 180$^\circ$ rotation (dual) of $P$, with the $w'$-labels of $P$ from \eqref{equ:bar} being unchanged under this rotation.

\begin{remark}  For the skew-diagram labeled posets of Example~\ref{exa:skew}, we see that switching strict and weak edges corresponds to transposing/conjugating the skew diagram $\lambda/\mu$. 
Therefore, the map that sends $\kpo$ to $\genx{\switch{P,\om}}$ extends the classical automorphism $\omega$ of the symmetric functions \cite[Exer.~7.94(a)]{ec2}.  (Here we use $\omega$ for two entirely different objects, with the hope that the distinction is clear from the context.)  Moreover, since skew Schur functions are invariant under 180$^\circ$ rotation of their skew diagrams, the map that sends $\kpo$ to $\genx{\switchx{\rotate{P,\om}}}$ also extends the automorphism $\omega$ \cite{Ges90,MalThesis, MaRe95}.   
\end{remark}

In view of Theorem~\ref{thm:kexpansion}, we can determine the effect of these involutions on $\kpo$ by examining their effect on the linear extensions of $(P,\om)$. 

\begin{lemma}\label{lem:involutions}  Let $(P,\om)$ be a labeled poset.
\begin{enumerate}
\item The descent sets of the linear extensions of $\switch{P,\om}$ are the complements of the descent sets of the linear extensions of $(P,\om)$.
\item The descent compositions of the linear extensions of $\rotate{P,\om}$ are the reverses of the descent compositions of the linear extensions of $(P,\om)$.  
\end{enumerate}
\end{lemma}

\begin{proof}
We first note that if labeled posets $(P_1, \om_1)$ and $(P_2, \om_2)$ are isomorphic, then it is clear that $(P_1,\om_1) \sim (P_2,\om_2)$ by Definition~\ref{def:kpo}, so Theorem~\ref{thm:kexpansion} implies that the descent sets of the linear extensions of $(P_1,\om_1)$ equal those of $(P_2,\om_2)$.   Therefore, for the purposes of this lemma, it suffices to consider $\switch{P,\om}$ and $\rotate{P,\om}$ up to isomorphism, allowing us to avail of \eqref{equ:barisom} and \eqref{equ:starisom}.

Since $\om(a) < \om(b)$ if and only if $\om'(a) > \om'(b)$, it is clear that the descent sets of the linear extensions of $(P,\om')$ are the complements of those of $(P,\om)$, implying (a).  For (b), observe that $(\pi(1), \ldots, \pi(n)) \in \mathcal{L}(P,\om)$ if and only if 
\[
(|P|+1-\pi(n), \ldots, |P|+1-\pi(1)) \in \mathcal{L}(P^*,\om').
\]
 Since $\pi(i) > \pi(i+1)$ if and only if $|P|+1-\pi(i+1) > |P|+1-\pi(i)$, we deduce (b).
\end{proof}

The usefulness of these involutions stems from the following consequence of Theorem~\ref{thm:kexpansion} and Lemma~\ref{lem:involutions}

\begin{proposition}\label{pro:fourfold}
For labeled posets $(P,\om)$ and $(Q,\tau)$, the following are equivalent:
\begin{itemize}
\item $(P,\om) \sim (Q,\tau)$;\medskip
\item $\switch{P,\om} \sim \switch{Q,\tau}$;\medskip
\item $\rotate{P,\om} \sim \rotate{Q,\tau}$;\medskip
\item $\switchx{\rotate{P,\om}} \sim \switchx{\rotate{Q,\tau}}$.
\end{itemize}
\end{proposition}


\section{Further necessary conditions}\label{sec:morenec}

In this section, we introduce a new tool for deriving necessary conditions for $\kpo$-equality, namely the \emph{jump sequence} of a labeled poset.  

\begin{definition}
We define the \emph{jump} of an element $b$ of a labeled poset $(P,\om)$ 
by considering the number of strict edges on each of the saturated chains from $b$ down to a minimal element of $P$, and taking the maximum such number.
The \emph{jump sequence} of $(P,\om)$, denoted $\jump(P,\om)$, is 
\[
\jump(P,\om) = (j_0, \ldots j_k),
\]
where $j_i$ is the number of elements with jump $i$, and $k$ is the maximum jump of an element of $(P,\om)$.  
\end{definition}

For example, the posets of Figure~\ref{fig:popartitions} both have jump sequence $(3,2)$.  The necessary condition for $\kpo$-equality is now as we might expect. 

\begin{proposition}\label{pro:jump}
If labeled posets $(P,\om)$ and $(Q,\tau)$ satisfy $(P,\om) \sim (Q,\tau)$, then $\jump(P,\om) = \jump(Q,\tau)$.  
\end{proposition}

\begin{proof}
Consider the \emph{greedy} or \emph{lexicographically greatest} $(P,\om)$-partition $g$, defined in the following way.  Choose $g$ so that $|g^{-1}(1)|$ is as large as possible and, from the remaining elements of $P$, $|g^{-1}(2)|$ is as large as possible, and so on.  We see that
\[
\jump(P,\om) = (|g^{-1}(1)|, \ldots, |g^{-1}(k)|)
\]
where $k$ is the maximum jump of an element of $(P,\om)$.  This sequence corresponds directly to $x_1^{|g^{-1}(1)|} \cdots x_k^{|g^{-1}(k)|}$, the monomial in the Definition~\ref{def:kpo} expansion of $\kpo$ with the lexicographically greatest sequence of exponents.  Since $\kpo = \kqt$, these two generating functions must have the same lexicographically greatest sequence of exponents among their monomials, so $\jump(P,\om) = \jump(Q,\tau)$. 
\end{proof}

Because of Proposition~\ref{pro:fourfold}, Proposition~\ref{pro:jump} actually gives four necessary conditions for $\kpo$-equality.  For example, we can compare labeled posets $(P,\om)$ according to $\jump(\switch{P,\om})$; the posets in Figure~\ref{fig:popartitions} have the same $\kpo$, as noted in Figure~\ref{fig:nontrivial}, and so have the same $\jump(\switch{P,\om}) = (2,2,1)$.  As another example, for a naturally labeled poset $(P,\om)$, $\jump(P,\om)=(|P|)$ does not really give useful information whereas looking at $\jump(\switch{P,\om})$ might allow us to say that two naturally labeled posets have different $\kpo$.  Along the same lines, we next give a corollary of Proposition~\ref{pro:jump} for the case of naturally labeled posets.  

\begin{corollary}\label{cor:natural}
  Let $(P,\om)$ and $(Q,\tau)$ be naturally labeled posets with $(P,\om) \sim (Q,\tau)$.  
\begin{enumerate}
\item $P$ and $Q$ have the same number of minimal elements, and the same number of maximal elements.  
\item $P$ and $Q$ have the same maximum chain length. 
\end{enumerate}
\end{corollary}

\begin{proof}
By Propositions~\ref{pro:fourfold} and~\ref{pro:jump} , we have $\jump(\switch{P,\om}) = \jump(\switch{Q,\tau})$.  The first entry in these jump sequences is the number of minimal elements.  The maximal-element analogue follows similarly from $\jump(\switchx{\rotate{P,\om}}) = \jump(\switchx{\rotate{Q,\tau}})$.

The length of the composition $\jump(\switch{P,\om})$ is the maximum chain length of $P$, and (b) follows.  
\end{proof}

Corollary~\ref{cor:natural}(a) is false for labeled posets in general: see Figure~\ref{fig:stathm} for a counterexample.  The analogue of Corollary~\ref{cor:natural}(a) for general labeled posets is that $(P,\om)$ and $(Q,\tau)$ have the same number of jump 0 elements, as do their three variations under the bar and star operations.  

\begin{question}\label{que:longestchain}
Is Corollary~\ref{cor:natural}(b) true when $(P,\om)$ and $(Q,\tau)$ are not necessarily naturally labeled?
\end{question}

The answer to Question~\ref{que:longestchain} is ``yes'' in the case of skew-diagram labeled posets, as can be seen from Lemma~8.4 and Corollary 8.11 of \cite{RSvW07}.  

We next consider necessary conditions involving the antichains of posets.  Recall that an \emph{antichain} of a poset $P$ is a subposet of $P$ all of whose elements are incomparable in $P$.  The \emph{width} of a poset is the cardinality of its largest antichain.  A \emph{convex subposet} $S$ of $P$ is a subposet of $P$ such that if $x,y,z \in P$ with $x <_P y <_P z$, then $x, z \in S$ implies $y \in S$.  A convex subposet of $(P,\om)$ inherits its designation of strict and weak edges from $(P,\om)$, and let us say that a convex subposet is \emph{strict} (resp.\ \emph{weak}) if all its edges are strict (resp.\ weak).  

\begin{proposition}\label{pro:convex}
If labeled posets $(P,\om)$ and $(Q,\tau)$ satisfy $(P,\om) \sim (Q,\tau)$, then the sizes of their largest weak (resp.\ strict) convex subposets are the same.
\end{proposition}

\begin{proof}  Observe that in $(P,\om)$, a subposet $S$ forms a weak convex subposet if and only if its elements can all have the same image in some $(P,\om)$-partition.  Therefore, again considering the monomials in $\kpo$, the size of the largest weak convex subposet of $(P,\om)$ is exactly the highest power of any variable $x_i$ that appears in any one of the monomials.  For the strict case, apply the same argument to $\switch{P,\om}$.
\end{proof}

Since all strict convex subposets of naturally labeled posets are antichains, we have the following corollary.

\begin{corollary}\label{cor:width}
If $(P,\om)$ and $(Q,\tau)$ are naturally labeled posets with $(P,\om) \sim (Q,\tau)$, then $P$ and $Q$ have the same width.
\end{corollary}

Corollary~\ref{cor:width} does not hold for general labeled posets, or even for skew-diagram labeled posets.  As shown in \cite[Example~4.1]{BTvW06}, the skew diagrams of Figure~\ref{fig:ribbons} have equal skew Schur functions, but we see that their corresponding labeled posets have widths 4 and 5 respectively.  
\begin{figure}[htbp]
\begin{center}
\begin{tikzpicture}[scale=0.4]
\begin{scope}
\draw[thick] (0,5) -- (1,5) -- (1,4) -- (2,4) -- (2,2) -- (4,2) -- (4,1) -- (5,1) -- (5,0) -- (3,0) -- (3,1) -- (1,1) -- (1,3) -- (0,3) -- cycle;
\draw (0,4) -- (1,4) -- (1,3) -- (2,3) 
(1,2) -- (2,2) -- (2,1)
(3,2) -- (3,1) -- (4,1) -- (4,0);
\end{scope}
\begin{scope}[xshift=8cm]
\draw[thick] (0,5) -- (1,5) -- (1,4) -- (3,4) -- (3,3) -- (4,3) -- (4,1) -- (5,1) -- (5,0) -- (3,0) -- (3,2) -- (2,2) -- (2,3) -- (0,3) -- cycle;
\draw (0,4) -- (1,4) -- (1,3) 
(2,4) -- (2,3) -- (3,3) -- (3,2) -- (4,2)
(3,1) -- (4,1) -- (4,0);
\end{scope}

\end{tikzpicture}
\caption{Two skew diagrams with equal skew Schur functions, but with different widths when converted to labeled posets.}
\label{fig:ribbons}
\end{center}
\end{figure}
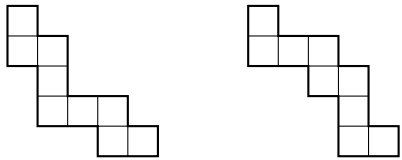

We extend Corollary~\ref{cor:width} with a conjecture that requires the following definition, which is reminiscent of the definition of a jump sequence.

\begin{definition}
Let the antichain sequence of a poset $P$, denoted $\anti(P)$, be the sequence $(a_1, \ldots, a_w)$, where $a_i$ is the number of antichains in $P$ of size $i$, and $w$ is the width of $P$.  
\end{definition}

\begin{conjecture}\label{con:antichain}
If naturally labeled posets $(P,\om)$ and $(Q,\tau)$ satisfy $(P,\om) \sim (Q,\tau)$, then $\anti(P) = \anti(Q)$.
\end{conjecture}

\begin{example}
Figure~\ref{fig:anti} shows two naturally labeled posets $P$ and $Q$ with the same generating function and with $\anti(P) = \anti(Q) = (7, 11, 3)$.  
\begin{figure}[htbp]
\begin{center}
\begin{tikzpicture}[scale=1.0]
\begin{scope}
\tikzstyle{every node}=[draw,shape=circle, inner sep=2pt];
\begin{scope}
\draw (1,0) node (a1) {}; 
\draw (2,0) node (a2) {}; 
\draw (0,1) node (a3) {};
\draw (1,1) node (a4) {};
\draw (2,1) node (a5) {};
\draw (1,2) node (a6) {};
\draw (2,2) node (a7) {};
\draw (a1) -- (a3)
(a1) -- (a4)
(a1) -- (a7)
(a2) -- (a5)
(a2) -- (a6)
(a3) -- (a6)
(a4) -- (a6)
(a5) -- (a7); 
\end{scope}
\begin{scope}[xshift=4cm]
\draw (1,0) node (a1) {}; 
\draw (2,0) node (a2) {}; 
\draw (0,1) node (a3) {};
\draw (1,1) node (a4) {};
\draw (2,1) node (a5) {};
\draw (1,2) node (a6) {};
\draw (2,2) node (a7) {};
\draw (a1) -- (a3)
(a1) -- (a5)
(a2) -- (a4)
(a2) to [bend right] (a7)
(a3) -- (a6)
(a4) -- (a6)
(a5) -- (a7); 
\end{scope}
\end{scope}
\begin{scope}[xshift=3cm,yshift=1cm]
\draw (0,0) node {$\sim$};
\end{scope}
\end{tikzpicture}
\caption{Two naturally labeled posets that have the same $\kpo$ and the same antichain sequence $(7,11,3)$.}
\label{fig:anti}
\end{center}
\end{figure}
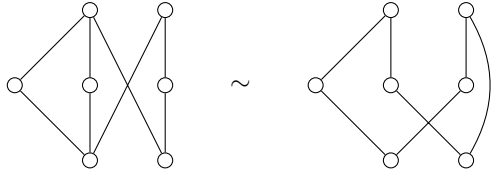
\end{example}

As a final observation of this section, let us note that the combination of all the necessary conditions above is not sufficient for $\kpo$-equality.  Figure~\ref{fig:notsuff} shows two (naturally) labeled posets $(P,\om)$ and $(Q,\tau)$ with $\kpo \neq \kqt$ (consider, for example, their coefficients of $M_{32}$) but for which the conclusions of Propositions~\ref{pro:cardinality}, \ref{pro:linexts}, \ref{pro:jump}, \ref{pro:convex} and Conjecture~\ref{con:antichain} are all satisfied by $(P,\om)$ and $(Q,\tau)$ and by the corresponding pairs under the bar and star operations.    

\begin{figure}[htbp]
\begin{center}
\begin{tikzpicture}[scale=1.0]
\begin{scope}
\tikzstyle{every node}=[draw,shape=circle, inner sep=2pt];
\begin{scope}
\draw (1,0) node (a1) {}; 
\draw (0,1) node (a2) {}; 
\draw (1,1) node (a3) {};
\draw (0,2) node (a4) {};
\draw (1,2) node (a5) {};
\draw (a4) -- (a2) -- (a5) -- (a3) -- (a1);
\end{scope}
\begin{scope}[xshift=3cm]
\draw (0,0) node (a1) {}; 
\draw (1,0) node (a2) {}; 
\draw (0,1) node (a3) {};
\draw (1,1) node (a4) {};
\draw (0,2) node (a5) {};
\draw (a5) -- (a3) -- (a1) -- (a4) -- (a2);
\end{scope}
\end{scope}
\end{tikzpicture}
\caption{Two labeled posets $(P,\om)$ and $(Q,\tau)$ which satisfy all our necessary conditions but for which $\kpo \neq \kqt$.}
\label{fig:notsuff}
\end{center}
\end{figure}
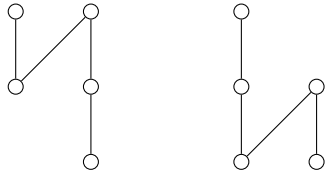


\section{Sufficient conditions}\label{sec:suff}

Most of our results in this section will take the following or similar form: if labeled posets $(P,\om)$ and $(Q,\tau)$ satisfy $(P,\om) \sim (Q,\tau)$, then $\widehat{(P,\om)} \sim \widehat{(Q,\tau)}$,
where $\widehat{(P,\om)}$ and $\widehat{(Q,\tau)}$ are modifications of $(P,\om)$ and $(Q,\tau)$ respectively.  

\subsection{Removing elements from labeled posets}

We begin with operations on labeled posets that preserve $\kpo$-equality and that decrease the number of elements in the labeled posets.  As in the previous section, we can apply the bar and star operations to get variations of these results.  In addition, for the next proposition, we could also restrict to naturally labeled posets to give nice statements involving minimal or maximal elements in place of jump 0 elements.

\begin{proposition}\label{pro:removal}
Suppose labeled posets $(P,\om)$ and $(Q,\tau)$ satisfy $(P,\om) \sim (Q,\tau)$.  Letting $\widehat{(P,\om)}$ and  $\widehat{(Q,\tau)}$ respectively denote the labeled posets $(P,\om)$ and $(Q,\tau)$ with their jump 0 elements removed, we have $\widehat{(P,\om)} \sim \widehat{(Q,\tau)}$.
\end{proposition}

\begin{proof}
In the expansion of $\kpo$ from Definition~\ref{def:kpo}, we know that the maximum power of $x_1$ which appears is the number $j_0$ of jump 0 elements of $(P,\om)$.  We also see that if we remove any monomials from $\kpo(x)$ that do not contain $x_1^{j_0}$, and then divide through by $x_1^{j_0}$, the result will be $\genx{\widehat{(P,\om)}}(x_2, x_3, \ldots)$.  Therefore, $\genx{\widehat{(P,\om)}}(x)$ can be uniquely obtained from $\kpo(x)$, and the result follows.
\end{proof}

\begin{example}
Starting with the labeled posets of Figure~\ref{fig:popartitions}, apply the bar and star operations to obtain the two new labeled posets of Figure~\ref{fig:removal}(a), each with just two elements of jump 0.  Removing these two elements from each labeled poset gives the labeled posets of Figure~\ref{fig:removal}(b) (cf.\ Figure~\ref{fig:stathm}), consistent with Proposition~\ref{pro:removal}.  
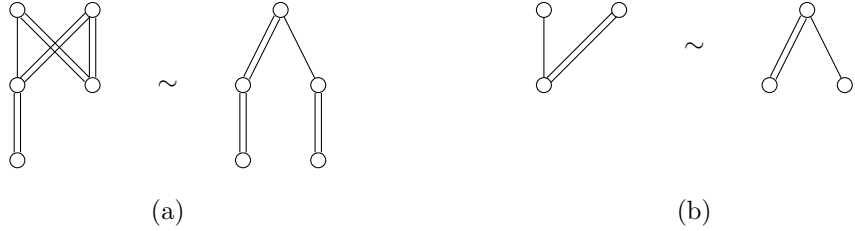
\begin{figure}[htbp]
\begin{center}
\begin{tikzpicture}[scale=1.0]
\begin{scope}
\tikzstyle{every node}=[shape=circle, inner sep=2pt]; 
\begin{scope}[rotate around={180:(0.5,1)}]
\draw (0,0) node[draw] (a1) {}; 
\draw (1,0) node[draw] (a2) {};
\draw (0,1) node[draw] (a3) {};
\draw (1,1) node[draw] (a4) {};
\draw (1,2) node[draw] (a5) {};
\draw (a2) -- (a4);
\draw[double distance=2pt] (a1) -- (a3)
(a2) -- (a3)
(a4) -- (a5); 
\draw (a1)++(0.092,0.035) -- + (0.87,0.87);
\draw (a1)++(0.035,0.092) -- + (0.87,0.87);
\end{scope}
\begin{scope}[xshift=3cm, rotate around={180:(0.5,1)}]
\draw (0.5,0) node[draw] (b1) {};
\draw (0,1) node[draw] (b2) {};
\draw (1,1) node[draw] (b3) {};
\draw (0,2) node[draw] (b4) {};
\draw (1,2) node[draw] (b5) {};
\draw (b1) -- (b2);
\draw[double distance=2pt]
(b1) -- (b3)
(b2) -- (b4)
(b3) -- (b5);
\end{scope}
\end{scope}
\begin{scope}[xshift=2cm,yshift=1cm]
\draw (0,0) node {$\sim$};
\end{scope}
\begin{scope}[xshift=2cm,yshift=-0.7cm]
\draw (0,0) node {(a)};
\end{scope}

\begin{scope}[xshift=7cm]
\begin{scope}
\tikzstyle{every node}=[shape=circle, inner sep=2pt]; 
\begin{scope}[rotate around={180:(0.5,1)}]
\draw (0,0) node[draw] (a1) {}; 
\draw (1,0) node[draw] (a2) {};
\draw (1,1) node[draw] (a4) {};
\draw (a2) -- (a4);
\draw (a1)++(0.092,0.035) -- + (0.87,0.87);
\draw (a1)++(0.035,0.092) -- + (0.87,0.87);
\end{scope}
\begin{scope}[xshift=3cm, rotate around={180:(0.5,1)}]
\draw (0.5,0) node[draw] (b1) {};
\draw (0,1) node[draw] (b2) {};
\draw (1,1) node[draw] (b3) {};
\draw (b1) -- (b2);
\draw[double distance=2pt]
(b1) -- (b3);
\end{scope}
\end{scope}
\begin{scope}[xshift=2cm,yshift=1.5cm]
\draw (0,0) node {$\sim$};
\end{scope}
\begin{scope}[xshift=2cm,yshift=-0.7cm]
\draw (0,0) node {(b)};
\end{scope}
\end{scope}
\end{tikzpicture}
\caption{An example of Proposition~\ref{pro:removal}}
\label{fig:removal}
\end{center}
\end{figure}
\end{example}

By induction, we get the following corollary of Proposition~\ref{pro:removal}.

\begin{corollary}
Suppose labeled posets $(P,\om)$ and $(Q,\tau)$ satisfy $(P,\om) \sim (Q,\tau)$.  Letting $\widehat{(P,\om)}$ and  $\widehat{(Q,\tau)}$ respectively denote the labeled posets $(P,\om)$ and $(Q,\tau)$ with their jump $i$ elements removed for all $i\leq k$ for any fixed $k$, we have $\widehat{(P,\om)} \sim \widehat{(Q,\tau)}$.
\end{corollary}

We conclude this subsection with a result which will be useful in Section~\ref{sec:fewlinext}; it has a stronger hypothesis but a stronger conclusion than Proposition~\ref{pro:removal}.

\begin{proposition}\label{pro:singlejump0}
Suppose labeled posets $(P,\om)$ and $(Q,\tau)$ both have only one jump 0 element.   Letting $\widehat{(P,\om)}$ and  $\widehat{(Q,\tau)}$ respectively denote the labeled posets $(P,\om)$ and $(Q,\tau)$ with this jump 0 element removed, we have $(P,\om) \sim (Q,\tau)$ if and only if $\widehat{(P,\om)} \sim \widehat{(Q,\tau)}$.
\end{proposition}

\begin{proof}
Since $(P,\om)$ has just one jump 0 element, every element of $\mathcal{L}(P,\om)$ has a descent in the first position, so there is a straightforward bijection from the descent compositions for the elements of $\mathcal{L}(P,\om)$ to the descent compositions for the elements of $\mathcal{L}(\widehat{(P,\om)})$.  Therefore, by Theorem~\ref{thm:kexpansion}, $\genx{\widehat{(P,\om)}}$ is uniquely determined by $\kpo$ and vice-versa in the following way: for any composition $\alpha$, each copy of $F_\alpha$ in $\genx{\widehat{(P,\om)}}$ corresponds to a copy of $F_{1,\alpha}$ in $\kpo$, and the result follows.  
\end{proof}

\subsection{Combining labeled posets}\label{sub:combining}

When one computes all the $\kpo$-equalities for $n=5$, one of the equivalences that appears is as shown in Figure~\ref{fig:fiveelements}.  Observe that it be can constructed in a natural way starting with the basic skew equivalence of Figure~\ref{fig:stathm}: for each of these two 3-element labeled posets $(P, \om)$, add a 2-element antichain $\{a,b\}$ below it, connecting $a$ (resp.\ $b$) to the minimal element(s) of $(P, \om)$ with weak (resp.\ strict) edges.  We next show that these operations preserve $\kpo$-equality.  To state our result in full generality will require some preliminary work.

\begin{figure}[htbp]
\begin{center}
\begin{tikzpicture}[scale=0.8]
\begin{scope}
\tikzstyle{every node}=[draw,shape=circle, inner sep=2pt];
\begin{scope}
\draw (0,0) node (a1) {1}; 
\draw (2,0) node (a2) {5}; 
\draw (1,1) node (a3) {3};
\draw (0,2) node (a4) {2};
\draw (2,2) node (a5) {4};
\draw (a1) -- (a3)
(a3) -- (a5);
\draw[double distance=2pt] (a2) -- (a3)
(a3) -- (a4);
\end{scope}
\begin{scope}[xshift=4cm]
\draw (0,0) node (a1) {1}; 
\draw (2,0) node (a2) {5}; 
\draw (0,1) node (a3) {2};
\draw (2,1) node (a4) {4};
\draw (1,2) node (a5) {3};
\draw[double distance=2pt] (a2) -- (a3)
(a2) -- (a4)
(a4) -- (a5);
\draw (a1) -- (a3)
(a1) -- (a4)
(a3) -- (a5);
\end{scope}
\end{scope}
\begin{scope}[xshift=3cm,yshift=1cm]
\draw (0,0) node {$\sim$};
\end{scope}
\end{tikzpicture}
\caption{An equivalence that can be built from the equivalence of Figure~\ref{fig:stathm}.}
\label{fig:fiveelements}
\end{center}
\end{figure}
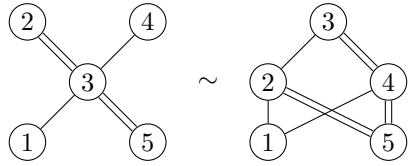

Let us interpret the labeled poset $(P,\om)$ drawn schematically in Figure~\ref{fig:addingposets}(a).  The poset consists of five disjoint convex subposets $P_1, \ldots, P_5$, each with some set of strict and weak edges that are determined by underlying labelings $\om_1, \ldots, \om_5$ respectively.    In addition, each of the minimal elements of $P_i$ for $i=1,2,3$ is connected to each of the maximal elements of $P_j$ for $j=4,5$ with all strict or all weak edges, with the type of edge as designated in the figure.  We allow any of the $P_i$ to be empty.

We could add further layers of posets to $(P,\om)$, but the proof given of Theorem~\ref{thm:addingposets} will be sufficient to show how one would prove an analogue for an even more complicated construction.  One word of warning: the construction shown in Figure~\ref{fig:addingposets}(b) is illegal since the result could not be a labeled poset.  Indeed, let $p_1$ and $p_2$ be minimal elements of $P_1$ and $P_2$ respectively, and let $p_3$ and $p_4$ be maximal elements of $P_3$ and $P_4$ respectively.  If a global labeling $\om$ existed, we would require the contradictory relations $\om(p_1) > \om(p_4) > \om(p_2) > \om(p_3) > \om(p_1)$.

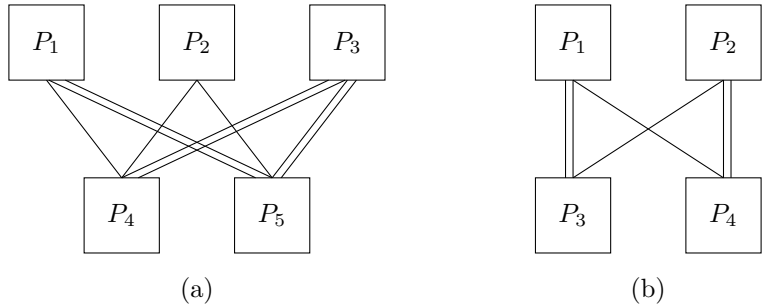
\begin{figure}[htbp]
\begin{center}
\begin{tikzpicture}[scale=1.0]
\draw(1,0) node (a1) {$P_4$}; 
\draw (a1)++(-0.5,-0.5) rectangle + (1,1);
\draw (3,0) node (a2) {$P_5$}; 
\draw (a2)++(-0.5,-0.5) rectangle + (1,1);
\draw (0,2.3) node (a3) {$P_1$};
\draw (a3)++(-0.5,-0.5) rectangle + (1,1);
\draw (2,2.3) node (a4) {$P_2$};
\draw (a4)++(-0.5,-0.5) rectangle + (1,1);
\draw (4,2.3) node (a5) {$P_3$};
\draw (a5)++(-0.5,-0.5) rectangle + (1,1);
\draw (a1)++(0,0.5) -- +(-1,1.3);
\draw (a1)++(0,0.5) -- +(1,1.3);
\draw (a2)++(0,0.5) -- +(-1,1.3);
\draw (a1)++(0,0.5) -- +(2.75,1.3);
\draw (a1)++(0.22,0.5) -- +(2.75,1.3);
\draw (a2)++(0,0.5) -- +(-2.75,1.3);
\draw (a2)++(-0.22,0.5) -- +(-2.75,1.3);
\draw (a2)++(0,0.5) -- +(1,1.3);
\draw (a2)++(0.12,0.5) -- +(1,1.3);
\draw (2,-1) node {(a)};
\begin{scope}[xshift=7cm]
\draw(0,0) node (a1) {$P_3$}; 
\draw (a1)++(-0.5,-0.5) rectangle + (1,1);
\draw (2,0) node (a2) {$P_4$}; 
\draw (a2)++(-0.5,-0.5) rectangle + (1,1);
\draw (0,2.3) node (a3) {$P_1$};
\draw (a3)++(-0.5,-0.5) rectangle + (1,1);
\draw (2,2.3) node (a4) {$P_2$};
\draw (a4)++(-0.5,-0.5) rectangle + (1,1);
\draw (a1)++(0,0.5) -- +(0,1.3);
\draw (a1)++(-0.1,0.5) -- + (0,1.3);
\draw (a1)++(0,0.5) -- +(2,1.3);
\draw (a2)++(0,0.5) -- +(-2,1.3);
\draw (a2)++(0,0.5) -- +(0,1.3);
\draw (a2)++(0.1,0.5) -- + (0,1.3);
\draw (1,-1) node {(b)};
\end{scope}
\end{tikzpicture}
\caption{(a) is the labeled poset $(P,\om)$ of Theorem~\ref{thm:addingposets}, while (b) shows an illegal construction.}
\label{fig:addingposets}
\end{center}
\end{figure}

\begin{theorem}\label{thm:addingposets}
Suppose we have labeled posets $(P_i, \om_i)$ and $(Q_i, \tau_i)$ for $i=1,\ldots,5$ such that $(P_i, \om_i) \sim (Q_i, \tau_i)$ for all $i$.  Construct the labeled poset $(P,\om)$ as shown in Figure~\ref{fig:addingposets}, and similarly construct $(Q,\tau)$.  Then $(P,\om) \sim (Q,\tau)$.
\end{theorem}

\begin{proof}
In view of Theorem~\ref{thm:kexpansion}, we wish to construct a descent-preserving bijection from $\mathcal{L}(P,\om)$ to $\mathcal{L}(Q,\tau)$.  Before doing this, we first need to define $\om$ in terms of the $\om_i$ in a way that respects the designation of strict and weak edges, both within the $P_i$ and as shown explicitly in Figure~\ref{fig:addingposets}(a).  One can check that the following definition of $\om(a)$ for $a \in P$ has the necessary properties:
\begin{gather}\label{equ:omdef}
\begin{aligned}
\om|_{P_3}(a) & = \om_3(a); \\
\om|_{P_4}(a) & = \om_4(a) + |P_3|;\\
\om|_{P_1}(a) & = \om_1(a) + |P_3| + |P_4|;\\
\om|_{P_5}(a) & = \om_5(a) + |P_3| + |P_4| + |P_1|;\\
\om|_{P_2}(a) & = \om_2(a) + |P_3| + |P_4| + |P_1| + |P_5|.
\end{aligned}
\end{gather}
In effect, the labels of each $P_i$ are preserved, but are shifted as necessary to respect the strict and weak edges that appear in Figure~\ref{fig:addingposets}(a).  The labeling $\tau$ will be defined similarly, with the letter $P$ replaced by the letter $Q$.  
Figure~\ref{fig:fiveelements} serves as an example of $(P,\om)$ and $(Q,\tau)$, where $P_2$ and $P_3$ are both empty, and $P_4$ and $P_5$ have just a single element each.

Now let $\pi \in \mathcal{L}(P,\om)$.  For $i=1,\ldots,5$, let $\pi_i$ correspond to the subword of $\pi$ consisting of $\om$-labels of elements of $P_i$.  Since $\gen{P_i, \om_i} = \gen{Q_i, \tau_i}$ for all $i$, there is a descent-preserving bijection from $\mathcal{L}(P_i,\om_i)$ to $\mathcal{L}(Q_i, \tau_i)$.  Therefore, for each $\pi_i$ there exists a canonical word $\rho_i$ consisting of $\tau$-labels of elements of $Q_i$ that uses the same letters as $\pi_i$ and has the same descent set as $\pi_i$.  Form a permutation $\rho$ by replacing each subword $\pi_i$ of $\pi$ with $\rho_i$ so that the letters of $\rho_i$ occupy the positions formerly occupied by the letters of $\pi_i$.  We claim that the map $\phi$ that sends $\pi$ to $\rho$ is a descent-preserving bijection from $\mathcal{L}(P,\om)$ to $\mathcal{L}(Q,\tau)$.  As an example, the permutation $\pi = 51342$ for the labeled poset on the left in Figure~\ref{fig:fiveelements} would yield $\rho=51243$.

We first show that $\rho \in \mathcal{L}(Q,\tau)$.  For $k \in [|Q|]$, we wish to show that the $\tau$-label $\rho(k)$ cannot appear above the $\tau$-label $\rho(k+1)$ in $(Q,\tau)$.  
First consider the case when the labels $\rho(k)$ and $\rho(k+1)$ both come from the same $Q_i$, as is the case for $k=3,4$ in our example.  By definition of $\rho$, $\rho(k)$ and $\rho(k+1)$ correspond to adjacent elements of $\rho_i$, which is itself a linear extension of $(Q_i,\tau_i)$.  Thus $\rho(k)$ and $\rho(k+1)$ respect the ordering in $(Q,\tau)$.  Now consider the case when the $\tau$-labels
$\rho(k)$ and $\rho(k+1)$ come from different $Q_i$.  This is the case for $k=1,2$ in our example.  Because of the construction of $(P,\om)$ in Figure~\ref{fig:addingposets}(a), $\pi$ will use up all the $\om$-labels from $P_4$ and $P_5$ before moving on to the $\om$-labels from $P_1$, $P_2$ and $P_3$.  By the definition of $\rho$, the same will be true for $\rho$ with the $Q_i$ in place of $P_i$, and $\tau$ in place of $\om$.   Therefore, when the $\tau$-labels $\rho(k)$ and $\rho(k+1)$ come from different $Q_i$, $\rho(k)$ cannot appear above $\rho(k+1)$ in $(Q,\tau)$.  Thus $\rho(k)$ and $\rho(k+1)$ again respect the ordering in $(Q,\tau)$, as required.

Because $\phi$ is defined in terms of the bijections from  $\mathcal{L}(P_i,\om_i)$ to $\mathcal{L}(Q_i, \tau_i)$, $\phi$ is itself a bijection.  Finally, we show that $\phi$ is descent-preserving.  Consider adjacent letters $\pi(k)$ and $\pi(k+1)$ of $\pi$.  If $\pi(k)$ and $\pi(k+1)$ are $\om$-labels from the same $P_i$ then, since the map from $\mathcal{L}(P_i,\om_i)$ to $\mathcal{L}(Q_i, \tau_i)$ is descent-preserving, position $k$ will be a descent in $\pi$ if and only if it is a descent in $\rho$.  This is the case for $k=3,4$ in our running example.  If $\pi(k)$ and $\pi(k+1)$ are $\om$-labels from $P_i$ and $P_j$ with $i\neq j$, then whether or not $k$ is a descent depends only on whether or not $\om|_{P_i}$ is defined before or after $\om|_{P_j}$ in \eqref{equ:omdef}.  Since the orderings of the $P_i$ in the definition of $\om$ is the same as the orderings of the $Q_i$ in the definition of $\tau$, $k$ will be a descent in $\pi$ if and only if it is a descent in $\rho$, as is the case for $k=1,2$ in our example.  
\end{proof}

\begin{remark}\label{rem:disjoint}
In the case when $P_i$ and $Q_i$ are empty for $i=3,4,5$, Theorem~\ref{thm:addingposets} shows that disjoint union of labeled posets preserves $\kpo$-equality.  This result also follows from Proposition~\ref{pro:product} which, in addition, tells us what $\genx{\du{P_1}{\om_1}{P_2}{\om_2}}$ equals.
\end{remark}

Along with the ways of combining posets of Theorem~\ref{thm:addingposets}, we can consider five additional ways of combining posets $P_1$ and $P_2$ in the case when both posets have unique minimal and/or unique maximal elements.  We wish to show that these combinations also preserve $\kpo$-equality.  The five ways are defined by Figure~\ref{fig:fiveways} and its caption, where we also set the notation for the five operations.  A node at the top (resp.\ bottom) of a labeled poset denotes that this poset has a unique maximal (resp.\ minimal) element.  In each case, the new labeled posets inherit strict and weak edges within $P_1$ and $P_2$ from $\om_1$ and $\om_2$ respectively.

\begin{remark}
A word of warning is in order here.  In certain situations, the strict and weak edges of $\nenw{P}{\om_1}{P_2}{\om_2}$ or $\NeNw{P}{\om_1}{P_2}{\om_2}$ might not be obtainable from an underlying global labeling $\om$, just like the illegal construction of Figure~\ref{fig:addingposets}.
For a specific example, consider $\NeNw{P}{\om_1}{P_2}{\om_2}$ in the case when both $P_1$ and $P_2$ are naturally labeled with at least one edge.  
However, this complication does not affect the validity of the proof of Theorem~\ref{thm:combinations} in the case when $\nenw{P}{\om_1}{P_2}{\om_2}$ or $\NeNw{P}{\om_1}{P_2}{\om_2}$ are valid labeled posets. 

Posets $P$ with an assignment of strict and weak edges that need not come from an underlying labeling are called \emph{oriented posets} in \cite{McN06}.   There are obvious analogues in the oriented poset case for the definition of $(P,\om)$-partitions in terms of strict and weak edges, for Definition~\ref{def:kpo}, and for the star and bar operations.  Notably, Proposition~\ref{pro:fourfold} carries through for oriented posets but Theorem~\ref{thm:kexpansion} does not.  Importantly, our proofs of Lemma~\ref{lem:addcomparability} and Theorem~\ref{thm:combinations} work equally well for oriented posets, so we choose not to include in the hypotheses that  $\nenw{P}{\om_1}{P_2}{\om_2}$ and $\NeNw{P}{\om_1}{P_2}{\om_2}$ be labeled posets.  See \cite{LaPy07,LaPy08,McN06} for more information on oriented posets.
\end{remark}

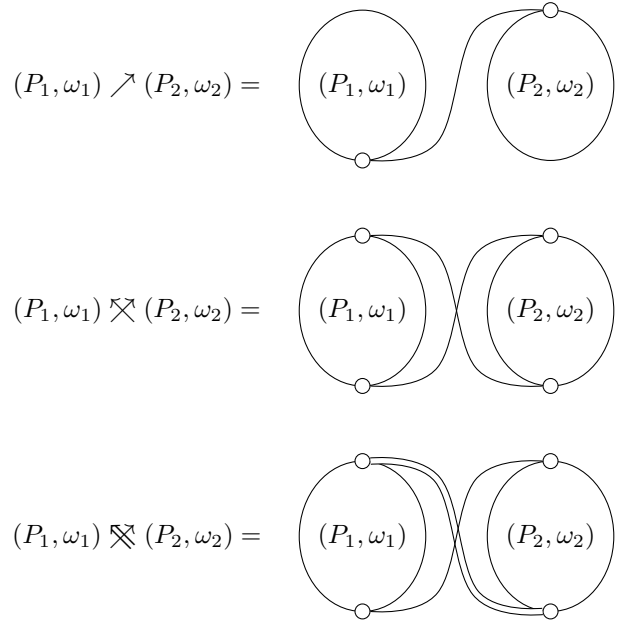
\begin{figure}[htbp]
\begin{center}
\begin{tikzpicture}[scale=1.0]
\begin{scope}
\begin{scope}
\draw (-3,2) node {$\neast{P_1}{\om_1}{P_2}{\om_2}=$};
\draw (0,2) node[ellipse,minimum height=2cm, minimum width=1.5cm,draw,inner sep=0pt] (P11) {$(P_1,\om_1)$};
\draw (0,1) node[circle, inner sep=2pt,fill=white,draw] (p11) {};
\end{scope}
\begin{scope}[xshift=2.5cm]
\draw (0,2) node[ellipse,minimum height=2cm, minimum width=1.5cm,draw,inner sep=0pt] (P21) {$(P_2,\om_2)$};
\draw (0,3) node[circle,inner sep=2pt,fill=white,draw] (p21) {};
\end{scope}
\draw plot[smooth, tension=0.65] coordinates{(p11.east) (1,1.25) (1.5,2.75) (p21.west)};
\end{scope}
\begin{scope}[yshift=-3cm]
\begin{scope}
\draw (-3,2) node {$\nenw{P_1}{\om_1}{P_2}{\om_2}=$};
\draw (0,2) node[ellipse,minimum height=2cm, minimum width=1.5cm,draw,inner sep=0pt] (P11) {$(P_1,\om_1)$};
\draw (0,1) node[circle,inner sep=2pt,fill=white,draw] (p11) {};
\draw (0,3) node[circle,inner sep=2pt,fill=white,draw] (p12) {};
\end{scope}
\begin{scope}[xshift=2.5cm]
\draw (0,2) node[ellipse,minimum height=2cm, minimum width=1.5cm,draw,inner sep=0pt] (P21) {$(P_2,\om_2)$};
\draw (0,1) node[circle,inner sep=2pt,fill=white,draw] (p22) {};
\draw (0,3) node[circle,inner sep=2pt,fill=white,draw] (p21) {};
\end{scope}
\draw plot[smooth, tension=0.65] coordinates{(p11.east) (1,1.25) (1.5,2.75) (p21.west)};
\draw plot[smooth, tension=0.65] coordinates{(p12.east) (1,2.75) (1.5,1.25) (p22.west)};
\end{scope}
\begin{scope}[yshift=-6cm]
\begin{scope}
\draw (-3,2) node {$\neNw{P_1}{\om_1}{P_2}{\om_2}=$};
\draw (0,2) node[ellipse,minimum height=2cm, minimum width=1.5cm,draw,inner sep=0pt] (P11) {$(P_1,\om_1)$};
\draw (0,1) node[circle,inner sep=2pt,fill=white,draw] (p11) {};
\draw (0,3) node[circle,inner sep=2pt,fill=white,draw] (p12) {};
\end{scope}
\begin{scope}[xshift=2.5cm]
\draw (0,2) node[ellipse,minimum height=2cm, minimum width=1.5cm,draw,inner sep=0pt] (P21) {$(P_2,\om_2)$};
\draw (0,1) node[circle,inner sep=2pt,fill=white,draw] (p22) {};
\draw (0,3) node[circle,inner sep=2pt,fill=white,draw] (p21) {};
\end{scope}
\draw[double distance=2pt] plot[smooth, tension=0.65] coordinates{(p12.east) (1,2.75) (1.5,1.25) 
(p22.west)};
\draw plot[smooth, tension=0.65] coordinates{(p11.east) (1,1.25) (1.5,2.75) (p21.west)};
\end{scope}
\end{tikzpicture}
\caption{Three operations on labeled posets $(P_1,\om_1)$ and $(P_2,\om_2)$.  
In addition, $\Ne{P}{\om_1}{P_2}{\om_2}$ is obtained by replacing the displayed weak edge  in the top figure with a strict edge, and $\NeNw{P}{\om_1}{P_2}{\om_2}$ is obtained from the middle figure by replacing both displayed weak edges with strict edges.}
\label{fig:fiveways}
\end{center}
\end{figure}

We need a preliminary observation, Lemma~\ref{lem:addcomparability}, which appears in \cite[Proof of Theorem~1]{Ges84} and in various contexts elsewhere in the literature.  First, we need to set some notation along the same lines as that of Figure~\ref{fig:fiveways}.  We are already using $+$ to denote disjoint union of labeled posets $(P_1,\om_1)$ and $(P_2,\om_2)$.  Ordinal sum of posets $P_1$ and $P_2$ is typically denoted $P_1 \oplus P_2$, and is defined by $x \leq_{P_1 \oplus P_2} y$ if (a) $x,y \in P_1$ and $x \leq_{P_1} y$, or (b) $x,y \in P_2$ and $x \leq_{P_2} y$, or (c) $x \in P_1$ and $y \in P_2$.  In terms of Hasse diagrams, $P_1$ appears below $P_2$ with an edge $e$ going from each maximal element of $P_1$ to each minimal element of $P_2$.    In our case, there will just be a single edge $e$ since $P_1$ (resp.\ $P_2$) will have a unique maximal (resp.\ minimal) element.  In the setting of labeled posets $(P_1,\om_1)$ and $(P_2,\om_2)$, we will need two different symbols depending on whether $e$ is strict or weak.  If $e$ is strict, we will write the resulting labeled poset as $\Up{P_1}{\om_1}{P_2}{\om_2}$, which also inherits strict and weak edges within $P_1$ and $P_2$ from $\om_1$ and $\om_2$ respectively.  In the case when $P_2$, $P_3$ and $P_4$ are empty, and $P_5$ (resp.\ $P_1$) has a unique maximal (resp.~minimal) element, Figure~\ref{fig:addingposets}(a) shows $\Up{P_5}{\om_5}{P_1}{\om_1}$.   
We will similarly use $\up{P_1}{\om_1}{P_2}{\om_2}$ when $e$ is weak.  

Following \cite{Ges84, StaThesis, ec1,ec2}, we will let $\mathcal{A}(P,\om)$ denote the set of $(P,\om)$-partitions for a labeled poset $(P,\om)$.  Equations~\eqref{equ:disjoint} and \eqref{equ:disjsum} below can be represented diagrammatically as in Figure~\ref{fig:disjsum}.  

\begin{figure}[htbp]
\begin{center}
\begin{tikzpicture}[scale=0.9]
\begin{scope}[xshift=-0.3cm]
\draw (0,2) node[ellipse,minimum height=2cm, minimum width=1.5cm,draw,inner sep=0pt] (P11) {$(P_1,\om_1)$};
\draw (0,0.94) node[ellipse, inner sep=1pt,fill=white,draw] (p11) {$p_1$};
\end{scope}
\begin{scope}[xshift=2cm]
\draw (0,2) node[ellipse,minimum height=2cm, minimum width=1.5cm,draw,inner sep=0pt] (P21) {$(P_2,\om_2)$};
\draw (0,3.06) node[ellipse,inner sep=1pt,fill=white,draw] (p21) {$p_2$};
\end{scope}
\draw (3.5,2) node {=};
\begin{scope}[xshift=5cm,yshift=-1.56cm]
\begin{scope}[yshift=3.12cm]
\draw (0,2) node[ellipse,minimum height=2cm, minimum width=1.5cm,draw,inner sep=0pt] (P11) {$(P_1,\om_1)$};
\draw (0,0.94) node[ellipse,inner sep=1pt,fill=white,draw] (p11) {$p_1$};
\end{scope}
\begin{scope}
\draw (0,2) node[ellipse,minimum height=2cm, minimum width=1.5cm,draw,inner sep=0pt] (P21) {$(P_2,\om_2)$};
\draw (0,3.06) node[ellipse,inner sep=1pt,fill=white,draw] (p21) {$p_2$};
\end{scope}
\draw[double distance=2pt]  (p11) -- (p21);
\end{scope}
\draw (6.5,2) node {+};
\begin{scope}[xshift=8cm]
\begin{scope}
\draw (0,2) node[ellipse,minimum height=2cm, minimum width=1.5cm,draw,inner sep=0pt] (P11) {$(P_1,\om_1)$};
\draw (0,0.94) node[ellipse,inner sep=1pt,fill=white,draw] (p11) {$p_1$};
\end{scope}
\begin{scope}[xshift=2.5cm]
\draw (0,2) node[ellipse,minimum height=2cm, minimum width=1.5cm,draw,inner sep=0pt] (P21) {$(P_2,\om_2)$};
\draw (0,3.06) node[ellipse,inner sep=1pt,fill=white,draw] (p21) {$p_2$};
\end{scope}
\draw plot[smooth, tension=0.65] coordinates{(p11.east) (1,1.25) (1.5,2.75) (p21.west)};
\end{scope}
\end{tikzpicture}
\caption{Equations \eqref{equ:disjoint} and \eqref{equ:disjsum} diagrammatically.  Note that $p_1$ and $p_2$ denote particular elements of $P_1$ and $P_2$, rather than $\om_1$ and $\om_2$ labels.}
\label{fig:disjsum}
\end{center}
\end{figure}
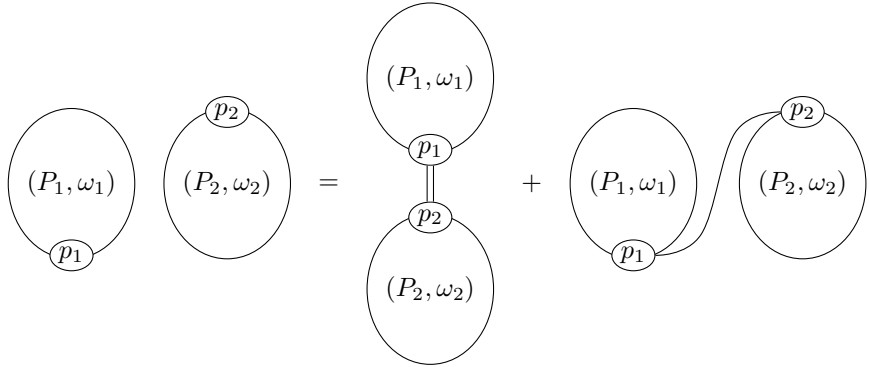

\begin{lemma}\label{lem:addcomparability}
Suppose a labeled poset $(P_1,\om_1)$ contains a unique minimal element $p_1$, and a labeled poset $(P_2,\om_2)$ contains a unique maximal element $p_2$.  
Then we have the disjoint union
\begin{equation}\label{equ:disjoint}
\mathcal{A}(\du{P_1}{\om_1}{P_2}{\om_2}) = \mathcal{A}(\Up{P_2}{\om_2}{P_1}{\om_1})\ \sqcup\ \mathcal{A}(\neast{P_1}{\om_1}{P_2}{\om_2}).
\end{equation}
Consequently,
\begin{equation}\label{equ:disjsum}
\genx{\du{P_1}{\om_1}{P_2}{\om_2}} = \genx{\Up{P_2}{\om_2}{P_1}{\om_1}} + \genx{\neast{P_1}{\om_1}{P_2}{\om_2}}\,.
\end{equation}
Similarly,
\[
\mathcal{A}(\du{P_1}{\om_1}{P_2}{\om_2}) = \mathcal{A}(\up{P_2}{\om_2}{P_1}{\om_1})\ \sqcup\ \mathcal{A}(\Ne{P_1}{\om_1}{P_2}{\om_2}),
\]
and
\begin{equation}\label{equ:disjsum2}
\genx{\du{P_1}{\om_1}{P_2}{\om_2}} = \genx{\up{P_2}{\om_2}{P_1}{\om_1}} + \genx{\Ne{P_1}{\om_1}{P_2}{\om_2}}\,.
\end{equation}
\end{lemma}

\begin{proof}
First observe that $p_1$ and $p_2$ are incomparable in $\du{P_1}{\om_1}{P_2}{\om_2}$.  Thus for a $(\du{P_1}{\om_1}{P_2}{\om_2})$-partition $\pop$, either $\pop(p_2) < \pop(p_1)$ or $\pop(p_2) \geq \pop(p_1)$.  The case $\pop(p_2) < \pop(p_1)$ is equivalent to $\pop \in \mathcal{A}(\Up{P_2}{\om_2}{P_1}{\om_1})$.  Similarly, the case $\pop(p_2) \geq \pop(p_1)$ corresponds to $\mathcal{A}(\neast{P_1}{\om_1}{P_2}{\om_2})$, implying \eqref{equ:disjoint}.   Definition~\ref{def:kpo} then gives \eqref{equ:disjsum}.  The remaining equations follow by considering instead the fact that either $\pop(p_2) \leq \pop(p_1)$ or $\pop(p_2) > \pop(p_1)$.  
\end{proof}

We are now ready to show that the five operations of Figure~\ref{fig:fiveways} preserve $\kpo$-equality.

\begin{theorem}\label{thm:combinations}  Suppose we have labeled posets $(P_1, \om_1)$, $(P_2,\om_2)$, $(Q_1, \tau_1)$ and $(Q_2, \tau_2)$, each of which has both a unique minimal and a unique maximal element.  If $(P_1,\om_1) \sim (Q_1,\om_1)$ and $(P_2,\om_2) \sim
(Q_2, \tau_2)$ then we have:
\begin{enumerate}
\item $\neast{P_1}{\om_1}{P_2}{\om_2} \ \sim\ \neast{Q_1}{\tau_1}{Q_2}{\tau_2}$;
\item $\Ne{P_1}{\om_1}{P_2}{\om_2}\ \sim\ \Ne{Q_1}{\tau_1}{Q_2}{\tau_2}$;
\item $\nenw{P_1}{\om_1}{P_2}{\om_2}\ \sim\ \nenw{Q_1}{\tau_1}{Q_2}{\tau_2}$;
\item $\NeNw{P_1}{\om_1}{P_2}{\om_2}\ \sim\ \NeNw{Q_1}{\tau_1}{Q_2}{\tau_2}$;
\item $\neNw{P_1}{\om_1}{P_2}{\om_2}\ \sim\ \neNw{Q_1}{\tau_1}{Q_2}{\tau_2}$.
\end{enumerate}
\end{theorem}

\begin{proof}
As noted in Remark~\ref{rem:disjoint}, disjoint union preserves $\kpo$-equality, i.e., 
\begin{equation}\label{equ:du}
\du{P_1}{\om_1}{P_2}{\om_2}\  \sim\  \du{Q_1}{\tau_1}{Q_2}{\tau_2}.
\end{equation}
By Theorem~\ref{thm:addingposets} with $P_2$, $P_3$ and $P_4$, empty, we know that 
the $\Uparrow$ operation preserves $\kpo$-equality, and so 
\begin{equation}\label{equ:Up}
\Up{P_2}{\om_2}{P_1}{\om_1}\  \sim\  \Up{Q_2}{\tau_2}{Q_1}{\tau_1}.
\end{equation}
Using \eqref{equ:du} and \eqref{equ:Up} in \eqref{equ:disjsum}, we deduce that 
\[
\neast{P_1}{\om_1}{P_2}{\om_2}\  \sim\  \neast{Q_1}{\tau_1}{Q_2}{\tau_2}
\]
and we have proved (a).  Similarly, (b) follows from \eqref{equ:disjsum2}.

For (c), consider $\neast{P_1}{\om_1}{P_2}{\om_2}$ as pictured on the right in Figure~\ref{fig:disjsum}, with $p_1$ and $p_2$ denoting the same elements as in the figure.  Let $p_1'$ denote the maximal element of $P_1$ and let $p_2'$ denote the minimal element of $P_2$.  Since $p_1'$ and $p_2'$ are incomparable, we can apply the ideas of the proof of Lemma~\ref{lem:addcomparability}: any $(\neast{P_1}{\om_1}{P_2}{\om_2})$-partition $\pop$ satisfies $\pop(p_2') \leq \pop(p_1')$ or $\pop(p_2') > \pop(p_1')$.  The former case is equivalent to $\pop$ being a $(\nenw{P_1}{\om_1}{P_2}{\om_2})$-partition.  The latter case is equivalent to $\pop$ being a $(\Up{P_1}{\om_1}{P_2}{\om_2})$-partition since the weak edge from $p_1$ to $p_2$ will be made redundant by the relations 
\[
\pop(p_1) \leq \pop(p_1') < \pop(p_2') \leq \pop(p_2).
\]
Putting this all together, we have
\[
\genx{\neast{P_1}{\om_1}{P_2}{\om_2}} = \genx{\nenw{P_1}{\om_1}{P_2}{\om_2}} + \genx{\Up{P_1}{\om_1}{P_2}{\om_2}}.
\]
Since the $\nearrow$ and $\Uparrow$ operations both preserve $\kpo$-equality, we conclude that the $\nenwarrows$ operation does too, implying (c).  To prove (d), we use (c), and exploit the bar operation of Proposition~\ref{pro:fourfold} by applying it to the equivalence 
\[
\switchnenw{P_1}{\om_1}{P_2}{\om_2}\ \sim\ \switchnenw{Q_1}{\tau_1}{Q_2}{\tau_2}.
\]

The proof of (e) is just like the proof of (c) except that we use the fact that any 
$(\neast{P_1}{\om_1}{P_2}{\om_2})$-partition $\pop$ satisfies $\pop(p_2') < \pop(p_1')$ or $\pop(p_2') \geq \pop(p_1')$.
\end{proof}

\begin{remarks} \
\begin{enumerate}
\item
From the proof, we see that (a) and (b) do not require the full hypotheses of the theorem, in that $P_1$ and $Q_1$ (resp.\ $P_2$ and $Q_2$) need not have unique maximal (resp.~minimal) elements.    
\item 
Theorem~5.9 does not give an exhaustive list of the ways in which two labeled posets can be combined in a manner that preserves $\kpo$-equality.  For example, we could combine $(P_1,\om_1)$ and $(P_2,\om_2)$ by drawing a weak edge from the unique minimum element of $P_1$ up to the unique \emph{minimum} element of $P_2$.  That this operation preserves $\kpo$-equality could be shown using the same method as in Theorem~\ref{thm:addingposets}.  The key requirement is a global labeling $\om$ that preserves the strict and weak edges, and the labeling $\om$ we defined for $\du{P_1}{\om_1}{P_2}{\om_2}$ (just before Proposition~\ref{pro:product}) would work in this case.  In fact, parts (a), (b) and (e) of Theorem~\ref{thm:combinations} could also have been proved using the technique of Theorem~\ref{thm:addingposets}.
\end{enumerate}
\end{remarks}


\section{Posets with few linear extensions}\label{sec:fewlinext}

In this section, we will classify all $\kpo$-equalities for posets with at most three linear extensions.  As we will see, all the equalities are built in a natural way from equalities among skew Schur functions.  

As we know from Proposition~\ref{pro:linexts}, if $(P,\om) \sim (Q,\tau)$, then  $|\mathcal{L}(P,\om)| = |\mathcal{L}(Q,\tau)|$.
So first, let us suppose that $(P,\om) \sim (Q,\tau)$ with $|\mathcal{L}(P,\om)|=1$.  Thus $P$ and $Q$ must both be chains, and Proposition~\ref{pro:jump} shows that $(P,\om) \cong(Q,\tau)$.

Next suppose that $(P,\om) \sim (Q,\tau)$ with $|\mathcal{L}(P,\om)|=2$.  We have seen one equivalence of this type in \eqref{equ:skewequiv} and Figure~\ref{fig:stathm}.  As we observed, this is the skew Schur equality $s_{21} = s_{22/1}$ translated to our labeled poset setting, and we will refer to it as the \emph{\stwoone\ equivalence.}  By repeatedly applying Theorem~\ref{thm:addingposets} in the case when just two of the $P_i$ are nonempty, we can build other equivalences with two linear extensions from the \stwoone\ equivalence.  For example, Figure~\ref{fig:twoexts} shows the equivalence
\begin{equation}\label{equ:chainext}
(\single \uparrow \sd{21} \uparrow \single \Uparrow \single)\ \ \sim\ \ (\single \uparrow \sd{22/1} \uparrow \single \Uparrow \single),
\end{equation}
where $\single$ denotes the labeled poset with one element, and $\sd{21}$ and $\sd{22/1}$ denote the labeled posets of Figure~\ref{fig:stathm}(a) and (b) respectively.  In Theorem~\ref{thm:fewlinexts} below, we will show that all $\kpo$-equalities with two linear extensions arise in this fashion from the \stwoone\ equivalence.  

\begin{figure}[htbp]
\begin{center}
\begin{tikzpicture}[scale=0.8]
\begin{scope}
\tikzstyle{every node}=[draw,shape=circle,inner sep=2pt]; 
\draw[line width=0pt] (1,0) node (a1) {}; 
\draw[line width=0pt] (0,1) node (a2) {}; 
\draw[very thick] (2,1) node (a3) {};
\draw[very thin] (1,-1) node (b1) {};
\draw (1,2) node (b2) {};
\draw (1,3) node (b3) {};
\draw[very thick] (a1) -- (a3);
\draw (b1) -- (a1)
(a2) -- (b2) -- (a3);
\draw[double distance=2pt, very thick](a1) -- (a2);
\draw[double distance=2pt] (b2) -- (b3);
\draw[very thick] (a1) node (a1cover) {};
\draw[very thick] (a2) node (a2cover) {};
\end{scope}
\begin{scope}[xshift=3.5cm,yshift=1cm]
\draw (0,0) node {$\sim$};
\end{scope}
\begin{scope}[xshift=5cm]
\begin{scope}
\tikzstyle{every node}=[draw,shape=circle, inner sep=2pt]; 
\draw[very thick] (0,0) node (a1) {}; 
\draw[line width=0pt] (2,0) node (a2) {}; 
\draw[line width=0pt] (1,1) node (a3) {};
\draw (1,-1) node (b1) {};
\draw (1,2) node (b2) {};
\draw (1,3) node (b3) {};
\draw[very thick] (a1) -- (a3);
\draw (a1) -- (b1) -- (a2)
(a3) -- (b2);
\draw[double distance=2pt, very thick](a2) -- (a3);
\draw[double distance=2pt] (b2) -- (b3);
\draw[very thick] (a2) node (a1cover) {};
\draw[very thick] (a3) node (a2cover) {};
\end{scope}
\end{scope}
\end{tikzpicture}
\caption{An example of a $\kpo$-equality with two linear extensions that is built from the \stwoone\ equivalence.}
\label{fig:twoexts}
\end{center}
\end{figure}
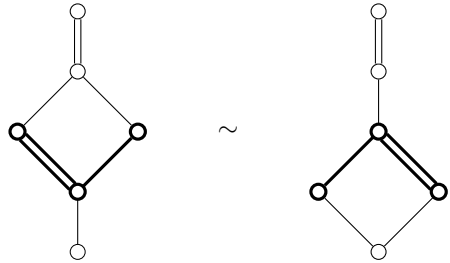

It turns out that the case of three linear extensions is very similar to the case of two linear extensions.  For three linear extensions, the basic building block is the $\sd{211}=\sd{222/11}$ equivalence shown in Figure~\ref{fig:threeexts}, which corresponds to the skew Schur equality $s_{211} = s_{222/11}$; we will refer to it as the \emph{\stwooneone\ equivalence}.

\begin{figure}[htbp]
\begin{center}
\begin{tikzpicture}[scale=0.8]
\begin{scope}
\tikzstyle{every node}=[draw,shape=circle, inner sep=2pt]; 
\draw (1,0) node (a1) {}; 
\draw (0,1) node (a2) {}; 
\draw (2,1) node (a3) {};
\draw (-1,2) node (a4) {};
\draw[double distance=2pt](a1) -- (a2) -- (a4);
\draw (a1) -- (a3);
\end{scope}
\begin{scope}[xshift=3.5cm, yshift=1cm]
\draw (0,0) node {$\sim$};
\end{scope}
\begin{scope}[xshift=5cm,yshift=1cm]
\begin{scope}
\tikzstyle{every node}=[draw,shape=circle, inner sep=2pt]; 
\draw (0,0) node (a1) {}; 
\draw (2,0) node (a2) {}; 
\draw (1,1) node (a3) {};
\draw (3,-1) node (a4) {};
\draw (a1) -- (a3);
\draw[double distance=2pt](a3) -- (a2) -- (a4);
\end{scope}
\end{scope}
\end{tikzpicture}
\caption{The \stwooneone\ equivalence, which is the basic building block for equivalences with three linear extensions.}
\label{fig:threeexts}
\end{center}
\end{figure}
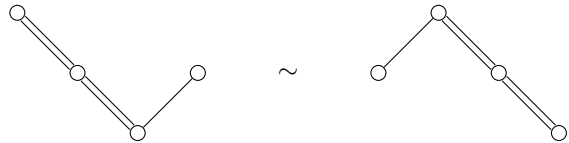

In the case of three linear extensions, Theorem~\ref{thm:addingposets} tells us that we can repeatedly apply any selection of the following operations to the \stwooneone\ equivalence to obtain other pairs of equivalent labeled posets, and it is clear that the new labeled posets will continue to have three linear extensions:
\begin{itemize}
\item add a bottom element to both labeled posets using a strict edge (or edges);
\item add a bottom element to both labeled posets using a weak edge (or edges);
\item add a top element to both labeled posets using a strict edge (or edges);
\item add a top element to both labeled posets using a weak edge (or edges).
\end{itemize}
An application of any number of these operations will be called a \emph{chain extension} of the  \stwooneone\ equivalence, since it has the effect of adding a chain with strict and/or weak edges to the bottom and/or top of the labeled posets in Figure~\ref{fig:threeexts}.  Chain extensions of other equivalences are defined in the same way.  In particular, chain extensions of the \stwoone\ equivalence have \eqref{equ:chainext} and Figure~\ref{fig:twoexts} as an example.

We can now state the main result of this section.

\begin{theorem}\label{thm:fewlinexts}
Suppose that labeled posets $(P,\om)$ and $(Q,\tau)$ satisfy $(P,\om) \sim (Q,\tau)$. 
\begin{enumerate}
\item If $|\mathcal{L}(P,\om)|=2$, then either $(P,\om) \cong (Q,\tau)$ or their equivalence is a chain extension of the \stwoone\ equivalence.  
\item If $|\mathcal{L}(P,\om)|=3$, then either $(P,\om) \cong (Q,\tau)$,  or their equivalence is a chain extension of the \stwooneone\ equivalence, or the equivalence of $\switch{P,\om}$ and $\switch{Q,\tau}$ is a chain extension of the \stwooneone\ equivalence.
\end{enumerate}
\end{theorem}

The extra clause in (b) doesn't appear in (a) since the posets in the \stwoone\ equivalence are preserved by the bar operation.  

The proofs of (a) and (b) are very similar, so we will prove them simultaneously.  We will need two lemmas about the structure of $(P,\om)$ and $(Q,\tau)$.  The first one will classify all labeled posets $(P,\om)$ with $|\mathcal{L}(P,\om)| \in \{2,3\}$.  For these purposes, we can ignore the labeling $\om$ and just classify all unlabeled posets $P$ with two or three linear extensions, i.e., $|\mathcal{L}(P)|\in \{2,3\}$.  Similar to the definition above of a chain extension of an equivalence, a \emph{chain extension of an unlabeled poset} $P$ is any poset that can be obtained by repeatedly adding top and/or bottom elements to $P$.  The $\mathbf{a}+\mathbf{b}$ poset is the $(a+b)$-element poset consisting of a disjoint union of a chain with $a$ elements and a chain with $b$ elements.   Once we ignore the strictness and weakness of the edges, both posets in Figure~\ref{fig:threeexts} are chain extensions of $\twoplusone$, while Figure~\ref{fig:twoexts} shows chain extensions of $\oneplusone$.  

\begin{lemma}\label{lem:classify3linexts}
For a poset $P$:
\begin{enumerate}
\item $|\mathcal{L}(P)|=2$ if and only if $P$ is a chain extension of $\oneplusone$;
\item $|\mathcal{L}(P)|=3$ if and only if $P$ is a chain extension of $\twoplusone$.  
\end{enumerate}

\end{lemma}

\begin{proof}
The ``if'' direction is clear.  For the ``only if'' direction, we see that if $P$ has an antichain of size 3 then $|\mathcal{L}(P)| \geq 6$.  Therefore $P$ has width 2 and so, by Dilworth's Theorem, the elements of $P$ can be partitioned into two chains $C_1$ and $C_2$.  Without loss of generality, we choose the partition so that $C_2$ has as few elements as possible.  Since we wish to show that $|C_2|=1$, let $a$ and $b$ be the minimum and maximum elements of $C_2$ respectively, and assume that $a \neq b$. Let $c$ be the minimum element in $C_1$ that is incomparable to $a$; we see that $c$ exists by the minimality of $|C_2|$.  Similarly, let $d$ be the maximal element of $C_1$ that is incomparable to $b$.  Clearly, $c \leq_P d$.  If $c=d$, this contradicts the minimality of $|C_2|$, since we could then partition $P$ in two chains such that the smaller chain only contained $c$.  

Therefore, the subposet $S$ of $P$ formed by $\{a,b,c,d\}$ has the relations $a<_S b$, $c<_S d$ and possibly the relations $c<_S b$ and $a<_S d$.  Importantly, $a$ and $c$ are incomparable, as are $b$ and $d$.  As a result, $S$ has at least four linear extensions, and hence so does $P$, a contradiction.  We conclude that $a=b$ and so $C_2$ has just a single element.  

With $c$ and $d$ defined as above, there are three possibilities.  The first possibility is that $c<_P d$ and there exists an element $z \in C_1$ with $c <_P z <_P  d$.  Then the subposet of $P$ formed by $\{a,c,z,d\}$ would have at least four linear extensions, and hence so would $P$, meaning that we can ignore this case.  The second possibility is that $c \pcovered d$, which gives that $|\mathcal{L}(P)|=3$ and we see that $P$ is indeed a chain extension of $\twoplusone$. 
The final possibility is that $c=d$, which is now possible because $a=b$; we have that $|\mathcal{L}(P)|=2$ and $P$ is indeed a chain extension of $\oneplusone$.  
\end{proof} 

We will prove one more preliminary result before proving Theorem~\ref{thm:fewlinexts}.  

\begin{lemma}\label{lem:uniqueminmax}
Suppose that labeled posets $(P,\om)$ and $(Q,\tau)$ satisfy $(P,\om) \sim (Q,\tau)$ and  
$|\mathcal{L}(P)| \in \{2,3\}$.  If $|P| = |Q| \geq 5$, then $P$ and $Q$ must both have unique minimal elements or both have unique maximal elements.
\end{lemma}

\begin{proof}
Suppose the result is false.  Then by Lemma~\ref{lem:classify3linexts}, up to switching $P$ and $Q$, they must have the unlabeled poset structures shown in Figure~\ref{fig:nouniqueminmax}, with the bold edge only appearing in the $|\mathcal{L}(P)|=3$ case. 
\begin{figure}[htbp]
\begin{center}
\begin{tikzpicture}[scale=0.6]
\begin{scope}
\tikzstyle{every node}=[draw,shape=circle, inner sep=2pt]; 
\draw (1,0) node (a1) {}; 
\draw (0,1) node (a2) {}; 
\draw (2,1) node (a3) {};
\draw[very thick] (-1,2) node (a4) {};
\draw (1,-1) node (l1) {};
\draw (1,-2) node (l2) {};
\draw (a1) -- (a3)
(a1) -- (a2)
(a1) -- (l1);
\draw[very thick] (a2) -- (a4);
\draw[thick, dotted] (l1) -- (l2);
\end{scope}
\begin{scope}[xshift=5cm,yshift=-1cm]
\begin{scope}
\tikzstyle{every node}=[draw,shape=circle, inner sep=2pt]; 
\draw (0,0) node (a1) {}; 
\draw (2,0) node (a2) {}; 
\draw (1,1) node (a3) {};
\draw[very thick] (3,-1) node (a4) {};
\draw (1,2) node (u1) {};
\draw (1,3) node (u2) {};
\draw (a1) -- (a3)
(a3) -- (a2)
(a3) -- (u1);
\draw[very thick] (a2) -- (a4);
\draw[thick, dotted] (u1) -- (u2);
\end{scope}
\end{scope}
\begin{scope}[xshift=-2cm]
\draw (0,0) node {$P=$};
\end{scope}
\begin{scope}[xshift=3.5cm]
\draw (0,0) node {$\sim$};
\end{scope}
\begin{scope}[xshift=9cm]
\draw (0,0) node {$=Q$};
\end{scope}
\end{tikzpicture}
\caption{The unlabeled poset structure if Lemma~\ref{lem:uniqueminmax} is false.  The bold edge is included if $|\mathcal{L}(P)|=3$ and is omitted if $|\mathcal{L}(P)|=2$.}
\label{fig:nouniqueminmax}
\end{center}
\end{figure}
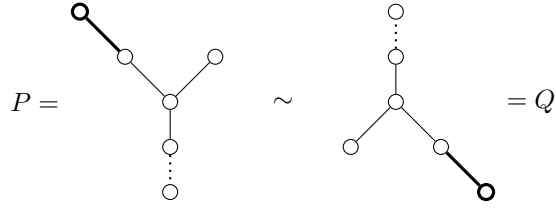
Let us compare $\mathcal{L}(P,\om)$ and $\mathcal{L}(Q,\tau)$.  Since $P$ has a unique minimum element and just one edge going up from that minimum element, every element of $\mathcal{L}(P,\om)$ will either have a descent in the first position, or no element of $\mathcal{L}(P,\om)$ will have a descent in the first position.  On the other hand, since $Q$ has two minimal elements, some elements of $\mathcal{L}(Q,\tau)$ will have a descent in the first position, and some will not.  Therefore, by Theorem~\ref{thm:kexpansion}, $\kpo \neq \kqt$, a contradiction.  
\end{proof}

\begin{proof}[Proof of Theorem~\ref{thm:fewlinexts}]
We proceed by induction on $n=|P|=|Q|$.  If $n\leq 5$, then the result holds by computationally checking all labeled posets; the only nontrivial equivalences with $\mathcal{L}(P,\om) \in \{2,3\}$ are the chain extensions of the \stwoone\ equivalence, and the chain extensions of the \stwooneone\ equivalence or of its bar-operation analogue.

Considering $n \geq 6$, we can apply Lemma~\ref{lem:uniqueminmax} to assume, applying the star operation if necessary, that $P$ and $Q$ both have a unique minimal element. 
We let $j_0$ denote the number of jump 0 elements in $(P,\om)$ which, by Proposition~\ref{pro:jump}, equals the number of jump 0 elements in $(Q,\tau)$.  The first case to consider is when $j_0=1$.  Letting $\widehat{(P,\om)}$ and  $\widehat{(Q,\tau)}$ respectively denote the labeled posets $(P,\om)$ and $(Q,\tau)$ with this jump 0 element removed, Proposition~\ref{pro:singlejump0} states that $ (P,\om) \sim (Q,\tau)$ if and only if $\widehat{(P,\om)} \sim \widehat{(Q,\tau)}$.
Since $(P,\om) \sim (Q,\tau)$ is a chain extension of $\widehat{(P,\om)} \sim \widehat{(Q,\tau)}$, the result follows by induction.  

It remains to consider $j_0 > 1$.  If all edges going up from the minimal element of both $(P,\om)$ and $(Q,\tau)$ are weak, we can apply the method of the previous paragraph to $\switch{P,\om}$ and $\switch{Q,\tau}$.  So let us assume, without loss of generality, that there is both a weak and a strict edge going up from the minimum element of $(Q,\tau)$.  Note that, by Lemma~\ref{lem:classify3linexts}, we cannot have more than two edges going up from the minimum element of $(Q,\tau)$ or of $(P,\om)$. 
If $(P,\om)$ has just one edge going up from its minimum element, then we see that $(P,\om) \sim (Q,\tau)$ does not hold by comparing the descents in the elements of $\mathcal{L}(P,\om)$ and $\mathcal{L}(Q,\tau)$, as in the proof of Lemma~\ref{lem:uniqueminmax}.  

By Lemma~\ref{lem:classify3linexts} and since $n\geq 6$, we conclude that $P$ and $Q$ must both have the unlabeled posets structure shown in Figure~\ref{fig:finaldeduction}(a) if $|\mathcal{L}(P)|=2$, and in Figure~\ref{fig:finaldeduction}(b) if $|\mathcal{L}(P)|=3$.  
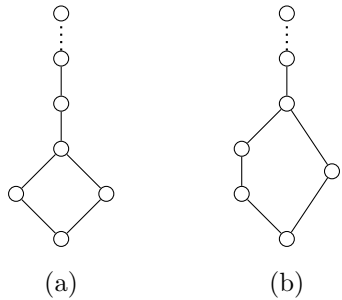
\begin{figure}[htbp]
\begin{center}
\begin{tikzpicture}[scale=0.6]
\begin{scope}
\tikzstyle{every node}=[draw,shape=circle, inner sep=2pt]; 
\draw (1,0) node (a1) {}; 
\draw (0,1) node (a2) {}; 
\draw (2,1) node (a4) {};
\draw (1,2) node (a5) {};
\draw (1,3) node (a6) {};
\draw (1,4) node (a7) {};
\draw (1,5) node (a8) {};
\draw (a1) -- (a2) -- (a5) -- (a6) -- (a7)
(a1) -- (a4) -- (a5);
\draw[thick, dotted] (a7) -- (a8);
\end{scope}
\begin{scope}[xshift=5cm]
\tikzstyle{every node}=[draw,shape=circle, inner sep=2pt]; 
\draw (1,0) node (a1) {}; 
\draw (0,1) node (a2) {}; 
\draw (0,2) node (a3) {};
\draw (2,1.5) node (a4) {};
\draw (1,3) node (a5) {};
\draw (1,4) node (a6) {};
\draw (1,5) node (a7) {};
\draw (a1) -- (a2) -- (a3) -- (a5) -- (a6)
(a1) -- (a4) -- (a5);
\draw[thick, dotted] (a6) -- (a7);
\end{scope}
\begin{scope}[xshift=1cm, yshift=-1cm]
\draw (0,0) node {(a)};
\end{scope}
\begin{scope}[xshift=6cm, yshift=-1cm]
\draw (0,0) node {(b)};
\end{scope}
\end{tikzpicture}
\caption{The two possibilities for the unlabeled poset structure of $P$ and $Q$ in the last paragraph of the proof of Theorem~\ref{thm:fewlinexts}.}
\label{fig:finaldeduction}
\end{center}
\end{figure}
Applying the star operation, $(P,\om)$ and $(Q,\tau)$ will both have just one edge going up from their unique minimum elements.  Since they both have the same number $j_0$ of jump 0 elements, we can apply the bar operation if necessary to ensure they both have $j_0=1$, and we are in a previously covered case.  
\end{proof}

It is natural to wonder about extending the results of this section to labeled posets with higher numbers of linear extensions.  The case $|\mathcal{L}(P)|=4$ already seems significantly more difficult since nothing as simple as Lemma~\ref{lem:classify3linexts} holds: there are at least two seemingly very different equivalences, namely the one shown in Figure~\ref{fig:fiveelements} and the equivalence that arises from $s_{2111} = s_{2222/111}$.  


\section{Open problems and concluding remarks}\label{sec:open}

Admittedly, a complete understanding of $\kpo$-equality is currently out of reach, which is not surprising given that the special case of skew Schur equality is wide open.  Our hope is that the study of $\kpo$-equality will be attractive to others because of the wealth of open problems in the area.  We have already presented Question~\ref{que:longestchain} and Conjecture~\ref{con:antichain}, and here we add some additional unanswered questions, along with some remarks. 

\subsection{Irreducibility of $\kpo$}
In the case of skew Schur function equality, it is sufficient to restrict one's attention to connected skew diagrams, as shown in \cite[\S 6]{RSvW07} and \cite{BRvW09}.  We have not been able to show that the same is true for the labeled poset case.

\begin{question}\label{que:disconnected}
Do there exist labeled posets $(P,\om)$, $(Q,\tau)$ and $(R, \upsilon)$ such that $P$ is connected and 
\[
\kpo = \kqt\, \gen{R,\upsilon}\ ?
\]
\end{question}

A positive answer to the next question would answer Question~\ref{que:disconnected} in the negative.

\begin{question}
Is $\kpo(x)$ irreducible in the ring of quasisymmetric functions over $\mathbb{Z}$ in infinitely many variables $x$? 
\end{question}

If the answer is ``yes,'' one might consider this related question.

\begin{question} Is $\kpo(x)$ irreducible in $\mathbb{Z}[[x]]$, the ring over $\mathbb{Z}$ of formal power series with bounded degree in infinitely many variables $x$?  
\end{question}

In \cite{LaPy08}, Lam and Pylyavskyy show some related results, including the irreducibility of $M_\alpha$ and $F_\alpha$ in both rings.  

\subsection{Sizes of equivalence classes}

In the case of skew Schur equality, it is conjectured in \cite{RSvW07} that all equivalence classes have a size that is a power of two.  This is false for general labeled posets, as there is an equivalence class of size 3 involving the labeled posets of Figure~\ref{fig:fiveelements}.   Applying the star operation to the labeled poset on the right of Figure~\ref{fig:fiveelements} yields a new labeled poset with the same $\kpo$ as the two shown, and there are no other labeled posets in this equivalence class.  

\subsection{Unexplained $\kpo$-equalities}
It is clear that we could use the results of Subsection~\ref{sub:combining} to construct equivalences involving arbitrarily large labeled posets.  On the other hand, the smallest labeled-poset equivalences that cannot be explained by our results occur when $|P|=5$.  Ignoring skew Schur equality, there are 16 $\kpo$-equalities when $|P|=5$, up to application of the bar and star operations.  Of these, all but four can be explained by our results.  These four are shown in Figure~\ref{fig:unexplained}, the first of which is recognizable from Figure~\ref{fig:nontrivial}.  We present them as a possible starting point for further investigation of $\kpo$-equality.

\begin{figure}[htbp]
\begin{center}
\begin{tikzpicture}[scale=1.0]
\begin{scope}
\begin{scope}
\tikzstyle{every node}=[shape=circle, inner sep=2pt]; 
\begin{scope}
\draw (0,0) node[draw] (a1) {};
\draw (1,0) node[draw] (a2) {};
\draw (0,1) node[draw] (a3) {};
\draw (1,1) node[draw] (a4) {};
\draw (1,2) node[draw] (a5) {};
\draw[double distance=2pt] (a2) -- (a4);
\draw (a1) -- (a3)
(a1) -- (a4)
(a2) -- (a3)
(a4) -- (a5); 
\end{scope}
\begin{scope}[xshift=3cm]
\draw (0.5,0) node[draw] (a1) {}; 
\draw (0,1) node[draw] (a2) {}; 
\draw (1,1) node[draw] (a3) {};
\draw (0,2) node[draw] (a4) {};
\draw (1,2) node[draw] (a5) {};
\draw[double distance=2pt] (a1) -- (a2);
\draw (a1) -- (a3)
(a2) -- (a4)
(a3) -- (a5);
\end{scope}
\end{scope}
\begin{scope}[xshift=2cm,yshift=1cm]
\draw (0,0) node {$\sim$}; 
\end{scope}
\end{scope}

\begin{scope}[yshift=-2cm]
\begin{scope}
\tikzstyle{every node}=[shape=circle, inner sep=2pt]; 
\begin{scope}[yshift=1cm, rotate=180]
\draw (-1,0) node[draw] (a1) {};
\draw (0,0) node[draw] (a2) {};
\draw (1,0) node[draw] (a3) {};
\draw (-0.5,1) node[draw] (a4) {};
\draw (0.5,1) node[draw] (a5) {};
\draw[double distance=2pt] (a1) -- (a4)
(a3) -- (a4);
\draw (a5)++(-1.2pt,-2.5pt) -- +(-1.36,-0.9)
(a5)++(-2.6pt,-0.5pt) -- +(-1.36,-0.9);
\draw (a2) -- (a4)
(a3) -- (a5)
(a2) -- (a5); 
\end{scope}
\begin{scope}[xshift=4cm]
\draw (-1,0) node[draw] (a1) {};
\draw (0,0) node[draw] (a2) {};
\draw (1,0) node[draw] (a3) {};
\draw (-0.5,1) node[draw] (a4) {};
\draw (0.5,1) node[draw] (a5) {};
\draw[double distance=2pt] (a1) -- (a4)
(a3) -- (a4);
\draw (a5)++(-1.2pt,-2.5pt) -- +(-1.36,-0.9)
(a5)++(-2.6pt,-0.5pt) -- +(-1.36,-0.9);
\draw (a2) -- (a4)
(a3) -- (a5)
(a2) -- (a5); 
\end{scope}
\end{scope}
\begin{scope}[xshift=2cm,yshift=0.5cm]
\draw (0,0) node {$\sim$}; 
\end{scope}

\end{scope}
\begin{scope}[yshift=-5cm]
\begin{scope}
\tikzstyle{every node}=[shape=circle, inner sep=2pt]; 
\begin{scope}
\draw (0,0) node[draw] (a1) {};
\draw (1,0) node[draw] (a2) {};
\draw (-0.5,1) node[draw] (a3) {};
\draw (0.5,1) node[draw] (a4) {};
\draw (1,2) node[draw] (a5) {};
\draw[double distance=2pt]  (a1) -- (a3)
(a2) -- (a5);
\draw (a4)++(0.2pt,2.8pt) -- +(0.42,0.84);
\draw (a4)++(2.3pt,1.5pt) -- +(0.42,0.84);
\draw (a1) -- (a4)
(a2) -- (a3); 
\end{scope}
\begin{scope}[xshift=3.5cm]
\draw (0,0) node[draw] (a1) {};
\draw (1,0) node[draw] (a2) {};
\draw (-0.5,1) node[draw] (a3) {};
\draw (0.5,1) node[draw] (a4) {};
\draw (1,2) node[draw] (a5) {};
\draw[double distance=2pt] (a2) -- (a3)
(a4)-- (a5);
\draw (a1) -- (a3)
(a1) -- (a4)
(a2) -- (a5);
\end{scope}
\end{scope}
\begin{scope}[xshift=2cm,yshift=1cm]
\draw (0,0) node {$\sim$}; 
\end{scope}
\end{scope}

\begin{scope}[yshift=-8cm]
\begin{scope}
\tikzstyle{every node}=[shape=circle, inner sep=2pt]; 
\begin{scope}
\draw (-0.5,0) node[draw] (a1) {};
\draw (0.5,0) node[draw] (a2) {};
\draw (0,1) node[draw] (a3) {};
\draw (1,1) node[draw] (a4) {};
\draw (0.5,2) node[draw] (a5) {};
\draw[double distance=2pt] (a2) -- (a3)
(a2) -- (a4)
(a4) -- (a5);
\draw (a1) -- (a3)
(a3) -- (a5); 
\end{scope}
\begin{scope}[xshift=3.5cm]
\draw (0.5,0) node[draw] (a1) {}; 
\draw (0,1) node[draw] (a2) {}; 
\draw (1,1) node[draw] (a3) {};
\draw (-0.5,2) node[draw] (a4) {};
\draw (0.5,2) node[draw] (a5) {};
\draw[double distance=2pt] (a1) -- (a2)
(a2) -- (a4);
\draw (a1) -- (a3)
(a2) -- (a5)
(a3) -- (a5);
\end{scope}
\end{scope}
\begin{scope}[xshift=2cm,yshift=1cm]
\draw (0,0) node {$\sim$}; 
\end{scope}
\end{scope}
\end{tikzpicture}
\caption{The four unexplained equivalences when $|P|=5$.}
\label{fig:unexplained}
\end{center}
\end{figure}
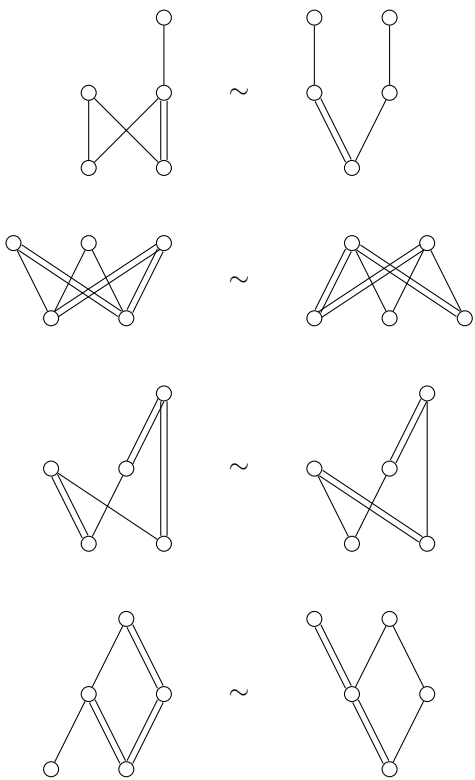

These unexplained $\kpo$-equalities suggest a further question.

\begin{question} What can be said about the labeled posets $(P,\om)$ that satisfy the following equivalences:
\begin{enumerate}
\item $(P,\om) \sim \rotate{P,\om}$, or \medskip
\item $(P,\om) \sim \switchx{\rotate{P,\om}}$ ?
\end{enumerate}
\end{question}

Examples of labeled posets $(P,\om)$ satisfying the equivalence of (a) include skew-diagram labeled posets, the labeled poset on the right in Figure~\ref{fig:fiveelements}, and the labeled posets of the second example in Figure~\ref{fig:unexplained}.  As for the equivalence of (b), examples include all labeled posets of Figure~\ref{fig:unexplained} except for those of the first equivalence, and all labeled posets obtained from self-transpose skew diagrams, such as the labeled posets of the \stwoone\ equivalence.  


\bibliography{ppartitionequality}
\bibliographystyle{alpha}

\end{document}